\documentclass[a4paper]{article}
\usepackage[english]{babel}
\usepackage{amsmath}
\title{The Morita-equivalence between MV-algebras and abelian $\ell$-groups with strong unit}
\author{Olivia Caramello\thanks{Supported by a CARMIN IH\'ES-IHP post-doctoral position (as from 1/12/2013) and by a visiting position of the Max Planck Institute for Mathematics (in the period 1/10/2013 - 30/11/2013).} \textrm{ }and Anna Carla Russo}
\date{April 22, 2014}

\usepackage[all]{xy}
\usepackage[T1]{fontenc}
\usepackage{tikz,color,makeidx, enumerate}
\usetikzlibrary{matrix,arrows}

\mathcode`\<="4268  
\mathcode`\>="5269  
\mathcode`\.="313A  
\mathchardef\semicolon="603B 
\mathchardef\gt="313E
\mathchardef\lt="313C

\newcommand{\cod}
 {{\rm cod}}

\newcommand{\comp}
 {\circ}

\newcommand{\Cont}
 {{\bf Cont}}

\newcommand{\dom}
 {{\rm dom}}

\newcommand{\empstg}
 {[\,]}

\newcommand{\epi}
 {\twoheadrightarrow}

\newcommand{\hy}
 {\mbox{-}}

\newcommand{\im}
 {{\rm im}}

\newcommand{\imp}
 {\!\Rightarrow\!}

\newcommand{\Ind}[1]
 {{\rm Ind}\hy #1}

\newcommand{\mono}
 {\rightarrowtail}
 
\newcommand{\name}[1]
 {\mbox{$\ulcorner #1 \urcorner$}}

\newcommand{\ob}
 {{\rm ob}}

\newcommand{\op}
 {^{\rm op}}

\newcommand{\Set}
 {{\bf Set }}

\newcommand{\Sh}
 {{\bf Sh}}

\newcommand{\sh}
 {{\bf sh}}

\newcommand{\Sub}
 {{\rm Sub}}

\usepackage{mathrsfs}
\usepackage{amsmath,amsthm}
\usepackage{amssymb}

\newtheorem{theorem}{Theorem}[section]
\theoremstyle{definition}
\newtheorem{definition}[theorem]{Definition}

\newtheorem{obs}[theorem]{Remark}
\newtheorem{obss}[theorem]{Remarks}
\newtheorem{pro}[theorem]{Proposition}
\newtheorem{cor}[theorem]{Corollary}
\newtheorem{remarks}[theorem]{Remarks}

\newcounter{commento}

\begin{document}

\maketitle
\begin{abstract}
We show that the theory of MV-algebras is Morita-equivalent to that of abelian $\ell$-groups with strong unit. This generalizes the well-known equivalence between the categories of set-based models of the two theories established by D. Mundici in 1986, and allows to transfer properties and results across them by using the methods of topos theory. We discuss several applications, including a sheaf-theoretic version of Mundici's equivalence and a bijective correspondence between the geometric theory extensions of the two theories.
\end{abstract}

\tableofcontents

\section{Introduction}

This paper is a contribution to the research programme `toposes as bridges' introduced in \cite{OC10}, which aims at developing the unifying potential of the notion of Grothendieck topos as a  means for relating different mathematical theories to each other through topos-theoretic invariants. The general methodology outlined therein is applied here to a duality of particular interest in the field of many-valued logics, namely Mundici's equivalence between the category of MV-algebras and that of abelian $\ell$-groups with a distinguished strong unit.   

The class of structures known as MV-algebras was introduced in 1958 by C. C. Chang (cf. \cite{chang1} and \cite{chang2}) in order to provide an algebraic proof of the completeness of \L ukasiewicz's multi-valued propositional logic. In the following years, the theory of MV-algebras has found applications in a variety of areas such as lattice-ordered abelian group theory (cf. \cite{Mundici1}, \cite{DiNolaLettieri}), functional analysis (cf. \cite{Mundici2}) and fuzzy set theory (cf. \cite{Belluce}).  

Notably, in 1986 (\cite{Mundici1} and \cite{Mundici2}) D. Mundici established an equivalence between the category of MV-algebras and that of lattice ordered abelian groups with a distinguished strong unit, extending Chang's equivalence (\cite{chang2}) between the category of MV-chains and that of totally ordered abelian groups with a distinguished strong unit.

In this paper we interpret Mundici's equivalence as an equivalence of categories of set-based models of two geometric theories, namely the algebraic theory of MV-algebras and the theory of abelian $\ell$-groups with strong unit, and show that this equivalence generalizes over an arbitrary Grothendieck topos yielding a Morita-equivalence between the two theories.

The fact that these two theories have equivalent classifying toposes - rather than merely equivalent categories of set-based models - has non-trivial consequences, some of which are explored in the final section of the paper (cf. also \cite{OC10} for a general overview of the significance of the notion of Morita-equivalence in Mathematics). For instance, the established Morita-equivalence provides, for any topological space $X$, an equivalence between the category of sheaves of MV-algebras over $X$ and the category of sheaves of abelian $\ell$-groups with unit over $X$ whose stalks are $\ell$-groups with strong unit, which is natural in $X$ and which extends Mundici's equivalence at the level of stalks. Another corollary is the fact that the geometric theory extensions of the two theories in their respective languages correspond to each other bijectively, a result that cannot be proved directly since - as we show in section \ref{interpretation} - the two theories are not bi-interpretable.   

The plan of the paper is as follows.  

In section 2 we fix the terminology and notation. 

In section 3 we introduce axiomatizations of the theory of MV-algebras and of that of abelian $\ell$-groups with strong unit within geometric logic, and describe the models of the two theories in an arbitrary Grothendieck topos.

In section \ref{s:generalization}, after reviewing the classical Mundici's equivalence, we proceed to generalizing it to the topos-theoretic setting, by defining, for every Grothendieck topos $\mathscr{E}$, two functors: one from the category of models of abelian $\ell$-groups with strong unit in $\mathscr{E}$ to the category for MV-algebras in $\mathscr{E}$, the other in the opposite direction. Further, we show that these two functors are categorical inverses to each other and that the resulting equivalence is natural in $\mathscr{E}$, thus concluding that the two theories are Morita-equivalent. 

In section \ref{s:applications} we present the two above-mentioned applications of the Morita-equivalence just built, and obtain a logical characterization of the $\ell$-groups corresponding to finitely presented MV-algebras. Specifically, we show that such groups can be characterized as the finitely presented abelian $\ell$-groups with unit whose unit is strong or, equivalently, as the $\ell$-groups presented by a formula which is irreducible relatively to the theory of abelian $\ell$-groups with strong unit (in the sense of section \ref{s:applications}). Moreover, in section \ref{logicprop} we show that the theory ${\mathbb L}_{u}$, whilst not finitary, enjoys a form of geometric compactness and completeness, by virtue of its Morita-equivalence with the finitary theory of MV-algebras.      

Section \ref{s:background} is an appendix in which we provide, for the benefit of non-specialists in topos theory, the basic topos-theoretic background necessary to understand the paper.

\section{Terminology and notation}\label{s:terminology and notation}

In this section we fix the relevant notation and terminology used throughout the paper.

All the toposes are Grothendieck toposes, and all the theories are geometric theories, if not otherwise specified (cf. section \ref{s:background} for the topos-theoretic background).

\begin{itemize}
\item[-] $\mathbb{MV}$, theory of MV-algebra;
\item[-] $\Sigma_{\mathbb{MV}}$, signature of the theory $\mathbb{MV}$;
\item[-] $\mathbb{L}_u$, theory of $\ell$-groups with stong unit;
\item[-] $\Sigma_{\mathbb{L}_u}$, signature of theory $\mathbb{L}_u$;
\item[-] $\mathbb{L}$, theory of $\ell$-groups with unit;
\item[-] $\mathbf{Set}$, category of sets and functions between sets;
\item[-] $\mathcal{MV}$, category of MV-algebras and MV-homomorphisms between them;
\item[-] $\mathcal{L}_u$, category of $\ell$-groups with strong unit and $\ell$-groups unital homomorphisms between them;
\item[-] $\mathbb{T}$-mod$(\mathscr{E})$, category whose objects are models of the theory $\mathbb{T}$ in the topos $\mathscr{E}$ and whose arrows are homomorphisms between them;
\item[-] $\mathscr{C}_{\mathbb T}$, geometric syntactic category of a geometric theory $\mathbb T$;
\item[-] $[[\vec{x}. \phi]]_{M}$, interpretation of a formula in a context $\vec{x}$ in a structure $M$;
\item[-] $\mathbf{Hom}_{geom}(\mathscr{C},\mathscr{D})$, category of geometric functors from a geometric category $\mathscr{C}$ to a geometric category $\mathscr{D}$;
\item[-] $\textrm{f.p.} {\mathbb T}\textrm{-mod}(\Set)$, category whose objects are the finitely presentable models of a geometric theory $\mathbb T$ (i.e., the models such that the hom functor $Hom_{{\mathbb T}\textrm{-mod}(\Set)}(M, -):{\mathbb T}\textrm{-mod}(\Set) \to \Set$ preserves filtered colimits) and whose arrows are the $\mathbb T$-model homomorphisms between them;
\item[-] $\Set[{\mathbb T}]$, classifying topos of a geometric theory $\mathbb T$;
\item[-] $\gamma_{\mathscr{E}}: \mathscr{E}\to \Set$, the unique (up to isomorphism) geometric morphism from a Grothendieck topos $\mathscr{E}$ to the category $\Set$ of sets.
\end{itemize}

\section{Axiomatizations}\label{s:syntax}

\subsection{Theory of MV-algebras $\mathbb{MV}$}
\begin{definition}\cite{CDM}
An \textit{MV-algebra} is an algebra $\mathcal{A}=(A,\oplus,\neg, 0)$ with a binary ope\-ration $\oplus$, an unary operation $\neg$ and a constant $0$, satisfying the following equations:
for any $x,y,z\in A$
\begin{enumerate}
\item $x\oplus(y\oplus z)=(x\oplus y)\oplus z$
\item $x\oplus y=y\oplus x$
\item $x\oplus 0= x$
\item $\neg \neg x=x$
\item $x\oplus \neg 0=\neg 0$
\item $\neg(\neg x\oplus y)\oplus y= \neg (\neg y\oplus x)\oplus x$
\end{enumerate}
\end{definition}

As a variety, the category of MV-algebras coincides with the category of models of a theory, which we denote by the symbol $\mathbb{MV}$. The signature $\Sigma_{\mathbb{MV}}$  of this theory consists of one sort, two function symbols and one constant:

\begin{itemize}
\item[] function symbols: $\oplus$ and $\neg$

\item[] constant symbol:  $0$
\end{itemize}
The axioms of $\mathbb{MV}$ are:
\begin{enumerate}
\item $\top \vdash_{x,y,z}x\oplus(y\oplus z)=(x\oplus y)\oplus z$
\item $(\top\vdash_{x,y}x\oplus y=y\oplus x)$
\item $(\top\vdash_{x}x\oplus 0=x)$
\item $(\top\vdash_{x}\neg \neg x=x)$
\item $(\top\vdash_{x}x\oplus \neg 0=\neg 0)$
\item $(\top\vdash_{x,y}\neg(\neg x\oplus y)\oplus y=\neg(\neg y\oplus x)\oplus x)$
\end{enumerate}

In any MV-algebra $\mathcal{A}$ there is a \textit{natural order} $\leq$ defined by: $x\leq y$ if and only if $\neg x\oplus y=\neg 0$ (for any $x,y\in A$); this is a partial order relation which induces a lattice structure on $\mathcal{A}$:
\begin{itemize}
\item[] $\sup(x,y):= (x\odot \neg y)\oplus y$;
\item[] $\inf(x,y):=x\odot(\neg x\oplus y)$,
\end{itemize} 
where $x\odot y:=\neg(\neg x\oplus \neg y)$.

We could thus enrich the signature of $\mathbb{MV}$ by introducing a `derived' relation symbol, $\leq$, and two `derived' operation symbols, $\sup$ and $\inf$.

The theory $\mathbb{MV}$ is clearly algebraic. A model of the theory $\mathbb{MV}$ in a category $\mathscr{E}$ with finite limits (in particular, a Grothendieck topos) consists of an object $M$, interpreting the unique sort of the signature $\Sigma_{\mathbb{MV}}$, an arrow $M\oplus : M\times M\rightarrow M$ in $\mathscr{E}$ interpreting the binary operation $\oplus$, an arrow $M\neg: M\rightarrow M$ in $\mathscr{E}$ interpreting the unary operation $\neg$ and a global element $M0:1\rightarrow M$ of $M$ in $\mathscr{E}$ (where $1$ is the terminal object of $\mathscr{E}$) interpreting the constant $0$. In the sequel, we shall omit the indication of the model $M$ in the notation for the operations and the constant.

An example of a model of the theory $\mathbb{MV}$ in a topos is the MV-algebra $([0,1],\oplus,\neg,0)$ in the topos $\mathbf{Set}$, where $x\oplus y=\inf(1,x+y)$ and $\neg x=1-x$ for any $x,y\in [0,1]$.

\subsection{Theory of abelian $\ell$-groups with strong unit $\mathbb{L}_u$}
\begin{definition}\cite{BK}
An abelian \textit{$\ell$-group with strong unit} is a structure $\mathcal{G}=(G,+,-,\leq ,\inf,\sup,0,u)$, where $(G,+,-,0)$ is an abelian group, $\leq$ is a partial order relation that induces a lattice structure and it is compatible with addition, i.e. it has the translation invariance property
\begin{center}
 $\forall x,y,t\in G$ \ $x\leq y\Rightarrow t+x\leq t+y$
\end{center}
and $u$ is an element of $G$ satisfying the property of being a strong unit, that is the conditions:
\begin{itemize}
\item[-] $u\geq 0$;
\item[-] for any $x\in G$, there is a natural number $n$ such that $x\leq nu$.
\end{itemize}
\end{definition}

The class of abelian $\ell$-groups with strong unit can be axiomatized as a geometric theory ${\mathbb L}_{u}$ (in the sense of section \ref{geom}). The signature $\Sigma_{\mathbb{L}_{u}}$ of ${\mathbb L}_{u}$ consists of one sort, four function symbols, one relation symbol and two constants:
\begin{itemize}
\item[]function symbols: $+$, $-$, $\sup$ and $\inf$
\item[]relation symbol: $\leq$
\item[]constants: $0$ and $u$
\end{itemize}
The axioms of $\mathbb{L}_{u}$ are:
\begin{enumerate}
\item $(\top\vdash_{x,y,z} x+(y+z)=(x+y)+z)$
\item $(\top\vdash_{x}x+0=x)$
\item $(\top\vdash_{x}x+(-x)=0)$
\item $(\top\vdash_{x,y}x+y=y+x)$
\item $(\top\vdash_{x}x\leq x)$
\item $((x\leq y)\wedge (y\leq x)\vdash_{x,y}x=y)$
\item $((x\leq y)\wedge (y\leq z)\vdash_{x,y,z}x\leq z)$
\item $(\top\vdash_{x,y} \inf(x,y)\leq x\wedge \inf(x,y)\leq y)$
\item $(z\leq x\wedge z\leq y\vdash_{x,y,z} z\leq \inf(x,y))$
\item $(\top\vdash_{x,y}(x\leq \sup(x,y) \wedge y\leq \sup(x,y)))$
\item $(x\leq z \wedge y\leq z\vdash_{x,y,z} \sup(x,y)\leq z)$
\item $(x\leq y\vdash_{x,y,t}t+x\leq t+y)$
\item $(\top\vdash u\geq 0)$
\item $(x\geq 0\vdash_{x} \bigvee_{n\in\mathbb{N}}x\leq nu)$
\end{enumerate}

Notice that the last axiom is provably equivalent to the sequent
\[
(\top \vdash_{x} \bigvee_{n\in\mathbb{N}}|x|\leq nu),
\] 
where the absolute value $|x|$ is the derived operation defined by $|x|=\sup(x, -x)$.

The Horn theory consisting of all the axioms above except for the last two is the theory of abelian $\ell$-groups with unit, and we denote it by the symbol $\mathbb L$. Of course, the theory $\mathbb L$ could be alternatively formalized as an algebraic theory, as the order relation $\leq$ can be defined for instance in terms of the operation $\inf$ (as $x\leq y$ if and only if $\inf(x, y)=x$).

A model of the theory ${\mathbb L}_{u}$ in a Grothendieck topos $\mathscr{E}$ consists of an object $N$ in $\mathscr{E}$, interpreting the unique sort of the signature of ${\mathbb L}_{u}$, and arrows in $\mathscr{E}$: $N+:N\times N\rightarrow N$, $N-:N\rightarrow N$, $N\sup:N\times N\rightarrow N$ and $N\inf :N\times N\rightarrow N$ interpreting the function symbols $+, -, \inf, \sup$, and a subobject $N\leq \rightarrowtail N\times N$ in $\mathscr{E}$ interpreting the order relation symbol $\leq$, plus global elements $N:1\rightarrow N$ and $Nu:1\rightarrow N$ interpreting the constants $0$ and $u$.

An example of a model of the theory $\mathbb{L}_{u}$ in a topos is the $\ell$-group with strong unit $(\mathbb{R},+,-,\leq,0,1)$ in $\mathbf{Set}$, where the operations $+,-$ and the relation $\leq$ are the usual ones.

\section{Generalization of Mundici's equivalence}\label{s:generalization}

In this section, after reviewing the classical Mundici's equivalence, we show how to generalize it to an arbitrary Grothendieck topos $\mathscr{E}$. First, we construct the functor $\Gamma_{\mathscr{E}}$ from the category of $\mathbb{L}_u$-models in $\mathscr{E}$ to the category of $\mathbb{MV}$-models in $\mathscr{E}$, and the functor $L_{\mathscr{E}}$ in the opposite direction. Then we show that they are categorical inverses to each other. The construction of these two functors being geometric, the equivalences between the categories of models in every Grothendieck topos will yield a Morita-equivalence between the theories.

\subsection{Mundici's equivalence in $\mathbf{Set}$}\label{s:mundici's equivalence}

In \cite{Mundici2} D. Mundici constructs a categorical equivalence between the category $\mathcal{MV}$ of MV-algebras and MV-homomorphisms between them and the category $\mathcal{L}_u$ of abelian $\ell$-groups with strong unit and unital abelian $\ell$-group homomorphisms between them. In terms of the theories $\mathbb{MV}$ and ${\mathbb L}_{u}$ defined in the last section, this is an equivalence between the categories $\mathbb{MV}$-mod$(\mathbf{Set})$ and $\mathbb{L}_u$-mod$(\mathbf{Set})$. 

Before generalizing this equivalence to an arbitrary Grothendieck topos, we first review the classical construction of the two functors defining it (cf. \cite{CDM} as a reference).

Let $\mathcal{G}=(G,+,-,\leq ,\inf,\sup,0,u)$ be an $\ell$-group with strong unit. We set 
\begin{center}
$\Gamma({\cal G}):=[0,u]=\{x\in G . 0\leq x\leq u\}$
\end{center}
and, for each $x,y\in [0,u]$,

\begin{center}
$x\oplus y:=\inf(u, (x+y)),$ \ \ \ \ \ $\neg x:= u-x$
\end{center} 
The structure $\Gamma(\mathcal{G}):=([0,u],\oplus,\neg,0)$ is the MV-algebra associated to the $\ell$-group $\mathcal{G}$.	

Any homomorphism $h$ of abelian $\ell$-groups with strong unit preservers the unit interval and, the operations $\oplus$ and $\neg$ being defined in term of the operations of $\mathcal{G}$, $h$ preserves them. Hence, the restriction of $h$ to the unit interval is an MV-algebra homomorphism and we can set $\Gamma(h):=h|_{[0,u]}$. Thus, we have a functor $\Gamma: \mathcal{L}_u\rightarrow \mathcal{MV}$.

In the converse direction, let $\mathcal{A}=(A,\oplus,\neg,0)$ be an MV-algebra. A sequence $\mathbf{a}=(a_{1},a_{2},...,a_n,...)$ of elements of $A$ is said to be \textit{good} if and only if $a_{i}\oplus a_{i+1}=a_{i}$, for each $i\in\mathbb{N}$, and there is a natural number $n$ such that $a_{r}=0$ for any $r\gt n$. For any pair of good sequences $\mathbf{a}$ and $\mathbf{b}$, one defines their sum $\mathbf{a}+\mathbf{b}$ as the sequence $\mathbf{c}$ whose components are
\begin{center}
 $c_{i}=a_{i}\oplus (a_{i-1}\odot b_{1})\oplus...\oplus (a_{1}\odot b_{i-1})\oplus b_{i}$.
\end{center}

Let $M_{A}$ be the set of good sequences of $\mathcal{A}$. The structure $M_{\mathcal{A}}:=(M_A,+,(0))$, where $(0)$ is the good sequence $(0, 0, ..., 0, ...)$, is an abelian monoid. The natural order in $\mathcal{A}$ induces a partial order relation $\leq$ in this monoid, given by: 
\begin{center}
$\mathbf{a}\leq \mathbf{b}$ if and only if $a_{i}\leq b_{i}$, for every $i\in \mathbb{N}$
\end{center}
Mundici proves, by using Chang's Subdirect Representation theorem (Theorem 1.3.3 \cite{CDM}), that this order admits inf and sup (for any pair of good sequences), which are given by:
\begin{center}
$\inf(\mathbf{a},\mathbf{b})=(\inf(a_1, b_1),...,\inf(a_n, b_n),...)$

\

$\sup(\mathbf{a},\mathbf{b})=(\sup(a_1, b_1),...,\sup(a_n, b_n),...)$
\end{center}

From the lattice abelian monoid $M_{\mathcal{A}}$ one can build an $\ell$-group ${\cal G}_{\mathcal{A}}$, by adding formal inverses to the elements of $M_{\mathcal{A}}$ (mimicking the construction of $\mathbb{Z}$ from $\mathbb{N}$). The elements of this $\ell$-group are equivalence classes $[x, y]$ of pairs of elements $x, y$ of the monoid. The constant $u=[(1),(0)]$, where $(1)=(\neg 0, 0, ..., 0, ...)$ is a strong unit for the group.

This construction is clearly functorial from the category $\mathcal{MV}$ to the category $\mathcal{L}_u$. The resulting functor $L:\mathcal{MV} \to \mathcal{L}_u$ is proved to be a categorical inverse to the functor $\Gamma$.

\subsection{From models of $\mathbb{L}_{u}$  to models of $\mathbb{MV}$}\label{interpretation}
Let $\mathscr{E}$ be a topos and $\mathcal{G}=(G,+,-,\leq ,\inf, \sup, 0, u)$ a model of $\mathbb{L}_{u}$ in $\mathscr{E}$. 

Mundici's construction of the functor $\Gamma:\mathbb{L}_u$-mod$(\mathbf{set})\rightarrow \mathbb{MV}$-mod$(\mathbf{Set})$ can be immediately generalized to any topos by using the internal language. Specifically, we define the interval $[0,u]$, where $u$ is the strong unit, as the subobject of $G$ 
\[
[0,u]:=[[x\in G. 0\leq x\leq u]],
\]
where the expression `$0\leq x\leq u$' is an abbreviation for the formula $(0\leq x) \wedge (x\leq u)$. 

We can define arrows 
\begin{center}
$\oplus: [0,u]\times [0,u]\rightarrow [0,u]$ 
 
 \
 
$\neg:[0,u]\rightarrow [0,u]$
\end{center}
in $\mathscr{E}$ again by using the internal language, as follows:

\begin{center}
$x \oplus  y=\inf(  u, x  + y),$ 

\

$\neg x= u  -x$.
\end{center}

The structure $\Gamma_{\mathscr{E}}(\mathcal{G}):=([0,u],\oplus, \neg, 0)$ is a model of the $\mathbb{MV}$ in $\mathscr{E}$ (c.f. Corollary \ref{models of MV}). In fact, the definition of the structure $\Gamma(\mathcal{G})$ with a first-order formula suggests that the theory $\mathbb{MV}$ is interpretable in $\mathbb{L}_u$. 

\begin{definition}
Let $\mathbb{T}$ and $\mathbb{S}$ be geometric theories. An \textit{interpretation} of $\mathbb{T}$ in $\mathbb{S}$ is a geometric functor $I:\mathscr{C}_{\mathbb{T}}\rightarrow \mathscr{C}_{\mathbb{S}}$ between their geometric syntactic categories. A \textit{bi-interpretation} is a (geometric) equivalence of categories. We say that $\mathbb{T}$ is \textit{interpretable} (resp. \textit{bi-interpretable}) in $\mathbb{S}$ if there exists an interpretation (resp. a bi-interpretation) of ${\mathbb T}$ in $\mathbb S$.
\end{definition} 

We recall from section D1.4 \cite{El} that for any geometric theory $\mathbb T$ and geometric category $\mathscr{C}$, we have a categorical equivalence 
\[
\mathbf{Hom}_{geom}(\mathscr{C}_{\mathbb{T}},\mathscr{C})\simeq \mathbb{T}\textrm{-mod}(\mathscr{C}),
\]
natural in $\mathscr{C}$, one half of which sends any model $M$ of $\mathbb{T}$ in $\mathscr{C}$ to the geometric functor $F_{M}:\mathscr{C}_{\mathbb{T}}\rightarrow\mathscr{C}$ assigning to any object $\{\vec{x}. \phi\}$ of the syntactic category $\mathscr{C}_{\mathbb{T}}$ its interpretation $[[\vec{x}. \phi]]_M$ in $\mathscr{C}$. 

Under this  equivalence, an interpretation $I$ of a theory $\mathbb{T}$ in a theory $\mathbb{S}$ corresponds to a model of $\mathbb{T}$ in the category $\mathscr{C}_{\mathbb{S}}$.

For any geometric category $\mathscr{C}$, an interpretation $I$ of ${\mathbb T}$ in $\mathbb S$ induces a functor
\[
s^{\mathscr{C}}_I:\mathbb{T}\textrm{-mod}(\mathscr{C}) \to \mathbb{S}\textrm{-mod}(\mathscr{C}) 
\] 
defined by the following commutative diagram:
\begin{center}
\begin{tikzpicture}
\node (1) at (0,0) {$\mathbf{Hom}_{geom}(\mathscr{C}_{\mathbb{T}},\mathscr{C})$};
\node (2) at (1.75,0) {$\simeq$};
\node (3) at (3,0) {$\mathbb{T}$-mod$(\mathscr{C})$};
\node (4) at (0,-3) {$\mathbf{Hom}_{geom}(\mathscr{C}_{\mathbb{S}},\mathscr{C})$};
\node (5) at (1.75,-3) {$\simeq$};
\node (6) at (3,-3) {$\mathbb{S}$-mod$(\mathscr{C})$};
\draw[<-] (4) to node [midway,left] {$-\circ I$} (1);
\draw[<-] (6) to node [midway,right] {$s^{\mathscr{C}}_I$} (3);
\end{tikzpicture}
\end{center}

\begin{theorem}\label{thm:interpretazione} 
The theory $\mathbb{MV}$ is interpretable in the theory $\mathbb{L}_u$, but not bi-interpretable.
\end{theorem}
\begin{proof}
As remarked above, to define an interpretation of $\mathbb{MV}$ in $\mathbb{L}_u$ it is equivalent to define a model of $\mathbb{MV}$ in the syntactic category $\mathscr{C}_{\mathbb{L}_u}$. Let us consider the object $A:=\{x. 0\leq x\leq u\}$ of $\mathscr{C}_{\mathbb{L}_u}$ and the following arrows in $\mathscr{C}_{\mathbb{L}_u}$:
\begin{itemize}
\item[-] $\oplus:=[x,y,z. z=\inf(u,x+y)]:A\times A\to A$;
\item[-] $\neg:=[x,z. z=u-x]:A\to A$;
\item[-] $0:=[x. x=0]:1\to A$
\end{itemize}
 
We have a $\Sigma_{\mathbb{MV}}$-structure in $\mathscr{C}_{\mathbb{L}_u}$, $\mathcal{A}=(A,\oplus,\neg,0)$.
The following sequents are provable in $\mathbb{L}_u$:
\begin{itemize}
\item[(i)]$(0\leq x,y,z\leq u\vdash_{x,y,z} x\oplus (y\oplus z)=(x\oplus y)\oplus z)$

\item[(ii)]$(0\leq x,y\leq u\vdash_{x, y} x\oplus y=y\oplus x)$

\item[(iii)]$(0\leq x\leq u\vdash_{x} x\oplus 0=0)$

\item[(iv)]$(0\leq x\leq u\vdash_{x} \neg\neg x=x)$

\item[(v)]$(0\leq x\leq u\vdash_{x} x\oplus \neg 0=\neg 0)$

\item[(vi)]$(0\leq x,y\leq u\vdash_{x, y} \neg(\neg x\oplus y)\oplus y=\neg(\neg y\oplus x)\oplus x)$
\end{itemize}

The proofs of these facts are straightforward. For instance, to prove sequent (ii), one observes that $x\oplus y=\inf (u,x+y)=\inf(u,y+x)=y\oplus x$, where the second equality follows from axiom 4 of the theory ${\mathbb L}_{u}$.

The validity of the axioms of the theory $\mathbb{MV}$ in the structure $\mathcal{A}$ is equivalent to provability of the sequents (i)-(vi) in the theory $\mathbb{L}_u$. Hence,  the structure $\mathcal{A}$ is a model of $\mathbb{MV}$ in $\mathscr{C}_{\mathbb{L}_u}$.

This proves that $\mathbb{MV}$ is interpretable in ${\mathbb L}_{u}$.

Suppose that there exists a bi-interpretation $J:\mathscr{C}_{\mathbb{L}_u}\rightarrow \mathscr{C}_{\mathbb{MV}}$.  
This induces a functor $s^{\Set}_J:\mathbb{MV}$-mod$(\mathbf{Set})\rightarrow \mathbb{L}_u$-mod$(\mathbf{Set})$ which is part of a categorical equivalence and which therefore reflects isomorphisms. 

Let $\mathcal{M}$ be an MV-algebra, $\mathcal{N}:=s^{\Set}_J(\mathcal{M})$ and $\{\vec{y}. \psi\}:=J(\{x. \top\})$. We have that $F_{\mathcal{N}}\simeq F_{\mathcal{M}}\circ J$, by the commutativity of the diagram preceding Theorem \ref{thm:interpretazione}
. Hence:
\begin{center}
$F_{\mathcal{N}}(\{x. \top\})\simeq F_{\mathcal{M}}(\{\vec{y}. \psi\})$

\

$[[x. \top]]_{\mathcal{N}}\simeq [[\vec{y}. \psi]]_{\mathcal{M}}$

\

$\mathcal{N}\simeq [[\vec{y}. \psi]]_{\mathcal{M}}$
\end{center}

If $\mathcal{M}$ is a finite MV-algebra then $[[\vec{y}. \psi]]_{\mathcal{M}}\subseteq \mathcal{M}^{n}$ (for some $n$) is finite as well; thus $\mathcal{N}$ is finite. By Corollary 1.2.13 \cite{BK}, every $\ell$-group is without torsion; hence, every non trivial $\ell$-group is infinite. It follows that $\mathcal{N}=s_J(\mathcal{M})$ is trivial for any finite MV-algebra $\mathcal{M}$. Since the functor $s^{\Set}_J$ reflects isomorphisms and there are two non-isomorphic finite MV-algebras, we have a contradiction.

\end{proof}

\begin{cor}\label{models of MV}
The structure $\Gamma_{\mathscr{E}}(\mathcal{G})$ is a model of $\mathbb{MV}$ in $\mathscr{E}$.
\end{cor}
\begin{proof}
By Theorem \ref{thm:interpretazione}, there is an interpretation  $I$ of $\mathbb{MV}$ in $\mathbb{L}_u$ and hence a functor $s_{I}^{\mathscr{E}}:\mathbb{L}_u$-mod$(\mathscr{E})\rightarrow \mathbb{MV}$-mod$(\mathscr{E})$. By definition of $I$, this functor sends any $\mathbb{L}_u$-model $\mathcal{G}$ to the structure $\Gamma(\mathcal{G})$. Hence $\Gamma(\mathcal{G})$ is a model of $\mathbb{MV}$ in $\mathscr{E}$.
\end{proof}

Let $h:\mathcal{G}\rightarrow \mathcal{G'}$ be a homomorphism between models of $\mathbb{L}_u$ in $\mathscr{E}$. Since $h$ preserves the unit and the order relation, it restricts to a morphism between the unit intervals $[0, u_{{\cal G}}]$ and $[0, u_{{\cal G'}}]$. This restriction is a MV-algebra homomorphism since $h$ clearly preserves the operations $\oplus$ and $\neg$ on these two algebras. Thus $\Gamma$ defines a functor from $\mathbb{L}_u$-mod$(\mathscr{E})$ to $\mathbb{MV}$-mod$(\mathscr{E})$.

\begin{obss}\label{remint}
\begin{enumerate}[(a)]
\item The interpretation functor $I$ defined above extends the assignment from MV-terms to $\ell$-group terms considered at p. 43 of \cite{CDM};

\item The functor $I$ sends every formula-in-context $\{\vec{x}. \phi\}$ in $\mathscr{C}_{\mathbb{MV}}$ to a formula in the same context $\vec{x}$ over the signature of ${\mathbb L}_{u}$. This can be proved by an easy induction on the structure of geometric formulae by using the fact that, by definition of $I$, for any formula-in-context $\phi(\vec{x})$ over the signature of $\mathbb{MV}$, the formula $I(\{\vec{x}. \phi\})$ is equal to the interpretation of the formula $\phi(\vec{x})$ in the internal MV-algebra $\mathcal{A}=(A,\oplus,\neg,0)$ in $\mathscr{C}_{\mathbb{L}_u}$ defined above. In particular, for any geometric sequent $\sigma=(\phi \vdash_{\vec{x}} \psi)$ over the signature of $\mathbb{MV}$ and any Grothendieck topos $\mathscr{E}$, the sequent $I(\sigma):=I(\{\vec{x}. \phi\}) \vdash_{\vec{x}} I(\{\vec{x}. \psi\})$ is valid in a $\ell$-group with unit $\cal G$ in $\mathscr{E}$ if and only if $\sigma$ is valid in the associated MV-algebra $[0, u_{\cal G}]$.
\end{enumerate}
\end{obss}

\subsection{From models of $\mathbb{MV}$ to models of $\mathbb{L}_{u}$}

More delicate is the generalization of the other functor of Mundici's equivalence which involves the concept of good sequence.

Let $\mathscr{E}$ be a Grothendieck topos, with its unique geometric morphism $\gamma_{\mathscr{E}}:\mathscr{E}\rightarrow \mathbf{Set}$ to the topos of sets. 

In $\mathbf{Set}$ the set of all sequences with values in a given set $A$ can be identified with the exponential $A^{\mathbb{N}}$ (where $\mathbb N$ is the set of natural numbers). This construction can be generalized to any topos; indeed, we can consider the object $A^{\gamma_{\mathscr{E}}^{*}(\mathbb{N})}$ in $\mathscr{E}$. As the functor $A^{-}: \mathscr{E}\rightarrow \mathscr{E}^{op}$ has a right adjoint, it preserves copro\-ducts. Therefore, since $\gamma_{\mathscr{E}}^{*}(\mathbb{N})=\bigsqcup_{n\in \mathbb{N}}\gamma_{\mathscr{E}}^{*}(1)$, the object $A^{\gamma_{\mathscr{E}}^{*}(\mathbb{N})}$ is isomorphic to $\prod_{n\in \mathbb{N}}A$, and $A^n \simeq A^{\gamma_{\mathscr{E}}^{*}(I_n)}$, where $I_n$ is the $n$-element set $\{1,...,n\}$. From this observation we see that the construction of the object of sequences $A^{\gamma_{\mathscr{E}}^{*}(\mathbb{N})}$ is not geometric; however, as we shall see below, the construction of the subobject of good sequences associated to an MV-algebra in a topos is geometric.

Let $\mathcal{A}=(A,\oplus, \neg, 0)$ be a model of $\mathbb{MV}$ in $\mathscr{E}$. We need to define the subobject of good sequences of $A^{\gamma^*_{\mathscr{E}}(\mathbb{N})}$.

We shall argue informally as we were working in the classical topos of sets, but all our constructions can be straightforwardly formalized in the internal language of the topos $\mathscr{E}$. 

\begin{definition}
We say that $\mathbf{a}=(a_{1},...a_{n})\in A^{n}$ is a \textit{$n$-good sequence} if
\begin{center}
$a_{i}\oplus a_{i+1}=a_{i}$, for any $i=1,...n-1$
\end{center}
\end{definition}

Let $s_n: S_n\rightarrow A^n$ be the subobject $\{\mathbf{a}\in A^n \mid \mathbf{a}$ is a $n$-good sequence of $A^n\}$ (for any $n\in\mathbb{N}$). 

Any $n$-good sequence can be completed to an infinite good sequence by setting all the other components equal to $0$. Anyway, $n$-good sequences for different natural numbers $n$ can give rise to the same infinite good sequence. Indeed, if $\mathbf{a}\in S_m$ and $\mathbf{b}\in S_n$, with $m\leq n$, are of the form $\mathbf{a}=(a_1,...,a_m)$ and $\mathbf{b}=(a_1,...,a_m,0,..,0)$ then the completed sequences coincide. 

This observation shows that we can realize the subobject of good sequences on $\cal A$ as a quotient of the coproduct $\bigsqcup_{n\in \mathbb{N}}S_{n}$ by a certain equivalence relation, which can be specified as follows (below we shall denote by $\chi_m: S_m\rightarrow \bigsqcup_{n\in \mathbb{N}}S_{n}$ the canonical coproduct injections). 

Consider, for each $m\leq n$, the arrow $\pi_{m,n}:A^{m}\rightarrow A^{n}$ which sends an $m$-sequence $\mathbf{a}$ to the $n$-sequence whose first $m$ components are those of $\mathbf{a}$ and the others are $0$. Notice that if $m=n$ then $\pi_{m,n}$ is the identity on $A^{m}$. As the image of a $m$-good sequence under $\pi_{m,n}$ is a $n$-good sequence, the arrows $\pi_{m,n}:A^m\rightarrow A^n$ restrict to the subobjects $s_m$ and $s_n$, giving rise to arrows: 

\begin{center}
 $\xi_{m,n}:S_{m}\rightarrow S_{n}$,
\end{center}
for each $m\leq n$.

Now we can define, by using geometric logic, a relation on the coproduct $\bigsqcup_{n\in \mathbb{N}}S_{n}$: for any $(\mathbf{a},\mathbf{b})\in \bigsqcup_{n\in \mathbb{N}}S_{n}\times \bigsqcup_{n\in \mathbb{N}}S_{n}$,
\begin{center}

$\mathbf{a}R\mathbf{b}:\Leftrightarrow \bigvee_{m\leq n} [(\exists \mathbf{a'}\in S_m)(\exists \mathbf{b'}\in S_n)(\chi_m(\mathbf{a'})=\mathbf{a}\wedge\chi_n(\mathbf{b'})=\mathbf{b}\wedge \xi_{m,n}(\mathbf{a'})=\mathbf{b'})]\bigvee_{n\leq m}[(\exists \mathbf{a'}\in S_m)(\exists \mathbf{b'}\in S_n)(\chi_m(\mathbf{a'})=\mathbf{a}\wedge \chi_n(\mathbf{b'})=\mathbf{b}\wedge \xi_{n,m}(\mathbf{b'})=\mathbf{a'})]$
\end{center}

It is immediate to check that this is an equivalence relation; in fact, $R$ can be characterized as the equivalence relation on the coproduct $\bigsqcup_{n\in \mathbb{N}}S_{n}$ generated by the family of arrows $\{\xi_{m,n} \mid m\leq n\}$. 

Let us now show how to realize the quotient $(\bigsqcup_{m\in\mathbb{N}}S_m)\slash R$, which is our candidate for the object of good sequences associated to the MV-algebra $\cal A$, as a subobject of the object $A^{\gamma^*_{\mathscr{E}}(\mathbb{N})}$ of `all sequences' on $A$.  

Let us define an arrow from $A^{\gamma^*_{\mathscr{E}}(I_m)}$ to $A^{\gamma^*_{\mathscr{E}}(\mathbb{N})}\cong A^{\gamma^*_{\mathscr{E}}(I_m)}\times A^{\gamma^*_{\mathscr{E}}(\mathbb{N}-I_m)}$ by setting the first component equal to the identity on $A^{\gamma^*_{\mathscr{E}}(I_m)}$ and the se\-cond equal to the composition of the unique arrow $A^{\gamma^*_{\mathscr{E}}(I_m)}\rightarrow 1$ with the arrow $0: 1\rightarrow A^{\gamma^*_{\mathscr{E}}(\mathbb{N}-I_m)}$, induced at each components by the zero arrow ${\cal A}0:1\to A$ of the MV-algebra $\mathcal{A}$. This arrow is clearly monic; by composing with $s_m:S_{m}\to A^{m}\cong A^{\gamma^*_{\mathscr{E}}(I_m)}$ we thus obtain a monomorphism $\nu_m:S_m\rightarrow A^{\gamma^*_{\mathscr{E}}(\mathbb{N})}$. These arrows determine, by the universal property of the coproduct, an arrow $\nu:\bigsqcup_{m\in\mathbb{N}} S_m\rightarrow A^{\gamma^*_{\mathscr{E}}(\mathbb{N})}$. This arrow clearly coequa\-lizes the two natural projections corresponding to the relation $R$; hence, we have a unique factorization $\nu\slash R:(\bigsqcup_{m\in\mathbb{N}}S_m)\slash R\rightarrowtail A^{\gamma^*_{\mathscr{E}}(\mathbb{N})}$ of $\nu$ through the quotient $(\bigsqcup_{m\in\mathbb{N}}S_m)\slash R$. This factorization is monic. Indeed, by using the internal language, if $[\mathbf{a}],[\mathbf{b}]\in (\bigsqcup_{m\in\mathbb{N}}S_m)_{/R}$ then there exist $m, n\in {\mathbb N}$, $\mathbf{a'}\in S_{m}$ and $\mathbf{b'}\in S_{n}$ such that $\mathbf{a}=\chi_{m}(\mathbf{a'})$ and $\mathbf{b}=\chi_{m}(\mathbf{b'})$; so $\nu(\mathbf{a})=\nu(\mathbf{b})$ if and only if $\nu_{m}(\mathbf{a'})=\nu_{n}(\mathbf{b'})$; but this clearly holds if and only if either $n\leq m$ and $\xi_{n,m}(\mathbf{b'})=\mathbf{a'}$ or $m\leq n$ and $\xi_{m,n}(\mathbf{a'})=\mathbf{b'}$, either of which implies that $\mathbf{a}R\mathbf{b}$ (i.e. $[\mathbf{a}]=[\mathbf{b}]$), as required.  

The subobject just defined admits natural descriptions in terms of internal language of the topos.

\begin{pro}\label{object good}
Let $\mathcal{A}=(A,\oplus, \neg, 0)$ be a model of $\mathbb{MV}$ in $\mathscr{E}$. Then the following monomorphisms to $A^{\gamma^*_{\mathscr{E}}(\mathbb{N})}$ are isomorphic:
\begin{itemize}
\item[(i)] $\nu\slash R:(\bigsqcup_{m\in\mathbb{N}}S_m)\slash R\rightarrowtail A^{\gamma^*_{\mathscr{E}}(\mathbb{N})}$;
\item[(ii)] $[[\mathbf{a}\in A^{\gamma_{\mathscr{E}}^*(\mathbb{N})}. \bigvee_{n\in {\mathbb N}} ((\exists \mathbf{a'}\in S_n)(\mathbf{a}=\chi_{m}(\mathbf{a'})))]]\rightarrowtail A^{\gamma^*_{\mathscr{E}}(\mathbb{N})}$.
\end{itemize}
We call the resulting subobject the \emph{subobject of good sequences} of the $\mathbb{MV}$-algebra $\cal A$, and denote it by the symbol $S_{\cal A}$. 
\end{pro}

\begin{proof}
The second subobject is, by the semantics of the internal language, the union of all the subobjects $\nu_m:S_m\rightarrow A^{\gamma^*_{\mathscr{E}}(\mathbb{N})}$. This union is clearly isomorphic to the image of the arrow $\nu$, which is isomorphic to $\nu\slash R$, as the latter arrow is monic and the canonical projection $(\bigsqcup_{m\in\mathbb{N}}S_m)\to  (\bigsqcup_{m\in\mathbb{N}}S_m)\slash R$ is epic. This proves the isomorphism between the first subobject and the second.
\end{proof} 

Notice that in the case $\mathscr{E}=\Set$, our definition of subobject of good sequences specializes to the classical one. 

Let us now proceed to define an abelian monoid structure on the object $S_{{\cal A}}$, by using the internal language of the topos $\mathscr{E}$.

Consider the term $\mathbf{a}\in A^{\gamma^*_{\mathscr{E}}(\mathbb{N})}$, and denote by $a_i$ the term $\mathbf{a}(\gamma^*_{\mathscr{E}}(\epsilon_i))$, where $\epsilon_i:1=\{\ast\}\rightarrow \mathbb{N}$ (for any $i\in \mathbb{N}$) is the function in $\mathbf{Set}$ defined by: $\epsilon_i(\ast):=i$ (the object $1$ is the terminal object in $\mathbf{Set}$). We can think of the $a_i$ as the \textit{components} of $\mathbf{a}$.

We set 
\begin{center}
$M_\mathcal{A}=(S_{{\cal A}},+,\leq,\sup, \inf,0)$,
\end{center}
where the operations and the relation are defined as follows (by using the internal language): for any $\mathbf{a},\mathbf{b}\in S_{{\cal A}}$, 
\begin{itemize}
\item  the sum $\mathbf{a}+\mathbf{b}$ is given by the sequence $\mathbf{c}\in S_{\cal A}$ whose components are $c_i:=a_i\oplus (a_{i-1}\odot b_i)\oplus ...\oplus (a_1\odot b_{i-1})\oplus b_i$;
\item $\sup(\mathbf{a},\mathbf{b})$, where $\sup(\mathbf{a},\mathbf{b})_i:=\sup(a_i,b_i)$;
\item $\inf(\mathbf{a},\mathbf{b})$, where $\inf(\mathbf{a},\mathbf{b})_i:=\inf(a_i,b_i)$;
\item $\mathbf{a}\leq \mathbf{b}$ if and only if $inf(\mathbf{a}, \mathbf{b})=\mathbf{a}$, equivalently if there exists $\mathbf{c}\in S_{{\cal A}}$ such that $\mathbf{a}+\mathbf{c}=\mathbf{b}$; 
\item $0=(0)$, i.e. $0_i=0$ for every $i\in \mathbb{N}$.
\end{itemize}

Mundici proves that this is an abelian lattice-ordered monoid in the case $\mathscr{E}$ equal to $\mathbf{Set}$. In fact, this is the case for an arbitrary $\mathscr{E}$.

\begin{pro}\label{monoide}
Let $\mathcal{A}$ be a model of $\mathbb{MV}$ in $\mathscr{E}$. Then $M_\mathcal{A}$ is a well-defined structure, i.e. all the operations are well-defined, and the axioms of the theory $\mathbb{L}_u$, except for the axiom 3, hold in $M_{\mathcal{A}}$. Furthermore, the structure $M_{\mathcal{A}}$ satisfies the cancellation property, i.e. if $\mathbf{a}+\mathbf{b}=\mathbf{a}+\mathbf{c}$ then $\mathbf{b}=\mathbf{c}$, for any $\mathbf{a},\mathbf{b},\mathbf{c}\in S_{\cal A}$.
\end{pro}
\begin{proof}

As shown in \cite{CDM}, all the required properties can be deduced from the validity of certain algebraic sequents written in the the signature of the theory $\mathbb{MV}$ for all MV-algebras in the given algebra $\cal A$. For instance, the associativity property can be deduced from the validity of the following sequent:
\[ 
\mathbf{a}\in S_n\wedge \mathbf{b}\in S_m\wedge \mathbf{c}\in S_k\vdash_{a_i,b_j,c_l}\bigwedge_{1\leq t\leq n+m+k} ((\mathbf{a}+\mathbf{b})+\mathbf{c})_t=(\mathbf{a}+(\mathbf{b}+\mathbf{c}))_t,
\]
where $((\mathbf{a}+\mathbf{b})+\mathbf{c})_t=d_t\oplus(d_{t-1}\odot c_1)\oplus...\oplus(d_1\odot c_{t-1})\oplus c_t,$

$(\mathbf{a}+(\mathbf{b}+\mathbf{c}))_t=a_t\oplus (a_{t-1}\odot f_1)\oplus...\oplus(a_1\odot f_{t-1})\oplus f_t$,

 $\mathbf{d}:=\mathbf{a}+\mathbf{b}$  \  \ \ $\mathbf{f}:=\mathbf{b}+\mathbf{c}$. 

These equalities can be easily verified to hold for all MV-chains (cf. \cite{CDM}) and can therefore be transferred to all MV-algebras thanks to Chang's subdirect representation theorem, if we assume our universe $\Set$ to satisfy the axiom of choice (which is necessary to prove the latter theorem). We thus deduce that these equalities are provable in the theory $\mathbb{MV}$ (by the classical completeness theorem for algebraic theories) and therefore valid in every $\mathbb{MV}$-algebra in a Grothendieck topos. 
\end{proof}

In order to make the given lattice-ordered abelian monoid into a lattice-ordered abelian group, we mimick the construction of $\mathbb Z$ from $\mathbb N$, as it is done in \cite{CDM}. Specifically, for any lattice-ordered abelian monoid $M$ satisfying the cancellation property in a topos $\mathscr{E}$, the corresponding lattice-ordered abelian group is obtained as the quotient of $M\times M$ by the equivalence relation $\sim$ defined, by using the internal language, as: $(a,b)\sim (c,d)$ if and only if $a+d=b+c$. The operations and the order on this structure are defined in the obvious well-known way, again by using the internal language. In particular, in the case of the $\ell$-group ${\cal G}_{\mathcal{A}}$ corresponding to the monoid ${\cal M}_{{\cal A}}$ associated to a MV-algebra $\cal A$, they are defined as follows: 

\begin{itemize}
\item \textit{addition}: $[\mathbf{a},\mathbf{b}]+[\mathbf{c},\mathbf{d}]:=[\mathbf{a}+\mathbf{c},\mathbf{b}+\mathbf{d}]$;
\item \textit{subtraction}: $-[\mathbf{a},\mathbf{b}]:=[\mathbf{b},\mathbf{a}]$;
\item $[\mathbf{a}, \mathbf{b}]\leq [\mathbf{c},\mathbf{d}]$ if and only if $\mathbf{a}+\mathbf{d}\leq \mathbf{c}+\mathbf{b}$, equivalently if and only if there exists $\mathbf{e}\in S_{\cal A}$ such that $[\mathbf{c},\mathbf{d}]-[\mathbf{a}, \mathbf{b}]=[\mathbf{e}, (0)]$;
\item $\sup([\mathbf{a},\mathbf{b}],[\mathbf{c},\mathbf{d}]):=[\sup(\mathbf{a}+\mathbf{d},\mathbf{c}+\mathbf{b}),\mathbf{b}+\mathbf{d}]$;
\item $\inf([\mathbf{a},\mathbf{b}],[\mathbf{c},\mathbf{d}]):=[\inf(\mathbf{a}+\mathbf{d},\mathbf{c}+\mathbf{b}),\mathbf{b}+\mathbf{d}]$;
\item \textit{zero element}: $[(0),(0)]\in G$, 
\end{itemize}
where the symbol $(0)$ indicates the sequence all whose components are zero.

Let us moreover define $u:=[(1),(0)]$, where the symbol $(1)$ indicates the sequence whose first component is $1$ and the others are $0$.

\begin{pro}
The structure ${\cal G}_{\mathcal{A}}$, equipped with the element $u:=[(1),(0)]$ as a unit, is a model of $\mathbb{L}_u$ in $\mathscr{E}$.
\end{pro}
\begin{proof}
We have already observed that ${\cal G}_{\mathcal{A}}$ is an $\ell$-group with unit. It remains to prove that $u$ is a strong unit in ${\cal G}_{\mathcal{A}}$. It is clear that $u\geq 0$. Reasoning with the internal language, if $0\leq [\mathbf{a},\mathbf{b}]\in G_{\cal A}$ then there exists $\mathbf{c}\in S_{\cal A}$ such that $[\mathbf{a},\mathbf{b}]$ is equal to $[\mathbf{c},(0)]$, thus there is a natural number $m$ such that $\mathbf{c}\in S_{m}$; then $m u=[1^m,(0)]\geq [\mathbf{c},(0)]$, where $1^m=(1,...,1, 0, 0, ..., 0,...)$ is the good sequence having the first $m$ components equal to $1$ and the others equal to $0$. 
\end{proof}

The assignment ${\cal A}\to {\cal G}_{\cal A}$ is clearly functorial; we thus obtain a functor

\begin{center}
$L_{\mathscr{E}}:\mathbb{MV}$-mod$(\mathscr{E})\rightarrow\mathbb{L}_u$-mod$(\mathscr{E})$
\end{center}
with $L_{\mathscr{E}}(\mathcal{A}):= \mathcal{G}_{\mathcal{A}}$ for any MV-algebra $\cal A$ in $\mathscr{E}$.

\subsection{Morita-equivalence}

In the previous sections we have built functors
\begin{center}
$L_{\mathscr{E}}:\mathbb{MV}$-mod$(\mathscr{E})\rightarrow\mathbb{L}_{u}$-mod$(\mathscr{E})$

\

$\Gamma_{\mathscr{E}}:\mathbb{L}_{u}$-mod$(\mathscr{E})\rightarrow\mathbb{MV}$-mod$(\mathscr{E})$
\end{center}
which generalize to an arbitrary topos $\mathscr{E}$ the classical functors forming Mundici's equivalence.

In the case $\mathscr{E}=\mathbf{Set}$ Mundici builds two natural transformations

\begin{itemize}
\item $\varphi:1_{\mathcal{MV}}\rightarrow \Gamma\circ L$, whose components are the arrows 
\begin{center}
$\varphi_{A}:A\rightarrow \Gamma\circ L(A)$

for any $a\in A$, \  $\varphi_{A}(a)=[(a),(0)]$,\ 
\end{center}
where $(a)$ denotes the sequence whose first component is $a$ and all the others $0$; 

\item $\psi:1_{\mathcal{L}_{u}}\rightarrow L\circ\Gamma$, whose components are the arrows 
\begin{center}
$\psi_{{\cal G}}:{\cal G}\rightarrow L\circ\Gamma({\cal G})$

for any $a\in G$,\ $\psi_{{\cal G}}(a)=[g(a^{+}), g(a^{-})]$
\end{center} 
where $a^{+}=\sup(a, 0)$, $a^{-}=\sup(-a, 0)$ and the expression $g(b)=(b_{1},...,b_{n},0,...)$, for an element $b\geq 0$ in $G$, denotes the unique good sequence with values $b_{i}$ in $\Gamma({\cal G})$ such that $b_{1}+...+b_{n}=b$.
\end{itemize}

\begin{pro}
For every $\mathcal{A}=(A,\oplus,\neg,0)\in \mathbb{MV}$-mod$(\mathscr{E})$, the arrows $\varphi_\mathcal{A}:a\in \mathcal{A}\rightarrow [(a),(0)]\in\Gamma_{\mathscr{E}}(\mathcal{G}_{\mathcal{A}})$ are isomorphisms natural in $\cal A$. In other words, they are the components of a natural isomorphism from the identity functor on $\mathbb{MV}$-mod$(\mathscr{E})$ to $\Gamma_{\mathscr{E}}\circ L_{\mathscr{E}}$.
\end{pro}
\begin{proof}
Let us argue in the internal language of the topos $\mathscr{E}$. The arrow $\phi_{\cal A}$ is clearly a monic homomorphism of MV-algebras. By definition of the order $\leq$ on ${\cal G}_{\cal A}$, we have that $[(0),(0)]\leq [\mathbf{a},\mathbf{b}]\leq [(1),(0)]$ if and only if there exists $c\in A$ such that $[\mathbf{a},(0)]=[(c),(0)]$. Hence, $\varphi_{\mathcal{A}}$ is epic. It is immediate to verify that $\varphi_{\mathcal{A}}$ preserves $\oplus$ and $\neg$, and that for any homomorphism $h:{\cal A}\to {\cal B}$ of MV-algebras in $\mathscr{E}$, the following square commutes:
\begin{center}
\begin{tikzpicture}
\node (1) at (0,0) {$\mathcal{A}$};
\node (2) at (5,0) {$\mathcal{B}$};
\node (3) at (0,-3) {$\Gamma_{\mathscr{E}}(L_{\mathscr{E}}(\mathcal{A}))$};
\node (4) at (5,-3) {$\Gamma_{\mathscr{E}}(L_{\mathscr{E}}(\mathcal{B}))$};
\draw[->] (1) to node [midway, above] {$h$} (2);
\draw[->] (3) to node [midway,below] {$\Gamma_{\mathscr{E}}(L_{\mathscr{E}}(h))$} (4);
\draw[->] (1) to node [midway,left] {$\varphi_{\mathcal{A}}$} (3);
\draw[->] (2) to node[midway,right] {$\varphi_{\mathcal{B}}$} (4);
\end{tikzpicture}
\end{center}
Thus, the arrows $\varphi_{\mathcal{A}}$ yield a natural isomorphism $\varphi:1\rightarrow \Gamma_{\mathscr{E}}\circ L_{\mathscr{E}}$.
\end{proof}

\begin{pro}
For every $\mathcal{G}=(G,+,-,\inf,\sup,0,u)\in \mathbb{L}_u$-mod$(\mathscr{E})$, there is an isomorphism $\phi_\mathcal{G}:L(\Gamma(\mathcal{G}))\rightarrow \mathcal{G}$, natural in $\mathcal{G}$. In other words, the isomorphisms $\phi_\mathcal{G}$ are the components of a natural isomorphism from $L_{\mathscr{E}}\circ \Gamma_{\mathscr{E}}$ to the identity functor on $\mathbb{L}_u$-mod$(\mathscr{E})$.   
\end{pro}
\begin{proof}

By using the internal language, we can generalize without problems the definition of the assignment $g$ (that is, Lemma 7.1.3 \cite{CDM}), and the proof that the map $\psi_{{\cal G}}$ it is injective and surjective. It remains to prove that $g$ is a homomorphism of abelian $\ell$-groups with unit. It will clearly be sufficient to prove that the inverse arrow

\begin{center}
$f_{\mathcal{G}}:(a_1,...a_n)\in M_{\Gamma(\mathcal{G})}\rightarrow a=a_1+...+a_n\in G^{+},$
\end{center}
where $G^{+}$ denotes the positive cone of $\cal G$, is a homomorphism, i.e. that the following properties hold:
\begin{itemize}
\item[(i)] $f_{\mathcal{G}}(\mathbf{a}+\mathbf{b})=f_{\mathcal{G}}(\mathbf{a})+f_{\mathcal{G}}(\mathbf{b})$;
\item[(ii)] $f_{\mathcal{G}}(\inf(\mathbf{a},\mathbf{b}))=\inf(f_{\mathcal{G}}(\mathbf{a}),f_{\mathcal{G}}(\mathbf{b}))$;
\item[(iii)] $f_{\mathcal{G}}(\sup(\mathbf{a},\mathbf{b}))=\sup(f_{\mathcal{G}}(\mathbf{a}),f_{\mathcal{G}}(\mathbf{b}))$;
\item[(iv)] $f_{\mathcal{G}}((u))=u$.
\end{itemize}
By definition of $f_{\mathcal{G}}$, property (iv) holds. Properties (i)-(iii) can be expressed in terms of the validity in the group $\cal G$ of certain sequents written in the signature of the theory $\mathbb{L}$ of $\ell$-groups with unit $u$. For instance, property (i) can be expressed by the sequent
\begin{itemize}
\item[]$(\mathbf{a}\in S_n\wedge \mathbf{b}\in S_m\bigwedge_{1\leq n}(0\leq a_i\leq u)\bigwedge_{1\leq j\leq m}(0\leq b_j\leq u)\vdash_{a_i,b_j}c_1+...+c_{n+m}=a_1+...+a_n+b_1+...+b_n),$
\end{itemize}
where $\mathbf{c}=(c_1,...,c_{n+m}):=\mathbf{a}+\mathbf{b}$.

Now, Mundici's proof of Lemma 7.1.5 \cite{CDM} shows that these sequents hold in any totally ordered abelian group with unit and hence, by Birkoff's classical result that every abelian $\ell$-group with unit is a subdirect product of totally ordered abelian groups, in every model of $\mathbb{L}$ in $\mathbf{Set}$, if one assumes the axiom of choice. The classical completeness theorem for cartesian theories thus allows us to conclude that these sequents are provable in $\mathbb L$ and hence valid in $\cal G$. 

Hence, $f_{\mathcal{G}}$ is a homomorphism, which induces  an isomorphism $\phi_{\mathcal{G}}:L_{\mathscr{E}}(\Gamma_{\mathscr{E}}(\mathcal{G}))\rightarrow\mathcal{G}$, for any $\mathcal{G}$. 

It is immediate to see that for any homomorphism $h:{\cal G}\to {\cal H}$ the square below commutes:
\begin{center}
\begin{tikzpicture}
\node (1) at (0,0) {$L_{\mathscr{E}}(\Gamma_{\mathscr{E}}(\mathcal{G}))$};
\node (2) at (6,0) {$L_{\mathscr{E}}(\Gamma_{\mathscr{E}}(\mathcal{H}))$};
\node (3) at (0,-3) {$\mathcal{G}$};
\node (4) at (6,-3) {$\mathcal{H}$};
\draw[->] (1) to node [midway, above] {$L_{\mathscr{E}}(\Gamma_{\mathscr{E}}(h))$} (2);
\draw[->] (3) to node [midway, below] {$h$} (4);
\draw[->] (1) to node [midway,left] {$\phi_{\mathcal{G}}$} (3);
\draw[->] (2) to node [midway,right] {$\phi_{\mathcal{H}}$} (4);
\end{tikzpicture}
\end{center}

We can thus conclude that the $\phi_{\mathcal{G}}$ are the components of a natural isomorphism from $L\circ \Gamma$ to the identity functor on $\mathbb{L}_u$-mod$(\mathscr{E})$, as required.
\end{proof}

We have built, for every Grothendieck topos $\mathscr{E}$, an equivalence of categories
\[
\mathbb{MV}\textrm{-mod}(\mathscr{E})\simeq \mathbb{L}_{u}\textrm{-mod}(\mathscr{F})
\]
given by functors

\[
L_{\mathscr{E}}:\mathbb{MV}\textrm{-mod}(\mathscr{E}) \to \mathbb{L}_{u}\textrm{-mod}(\mathscr{F})
\]
and
\[
\Gamma_{\mathscr{E}}:\mathbb{L}_{u}\textrm{-mod}(\mathscr{F} \to \mathbb{MV}\textrm{-mod}(\mathscr{E})
\]
generalizing the classical functors of Mundici's equivalence. 

To prove that the theories $\mathbb{MV}$ and ${\mathbb L}_{u}$ are Morita-equivalent, it remains to show that this equivalence is natural in $\mathscr{E}$, that is for any geometric morphism $f:\mathscr{F}\to \mathscr{E}$, the following diagrams commute:

\begin{center} 
\begin{tikzpicture}
\node (1) at (0,0) {$\mathbb{L}_{u}$-mod$(\mathscr{F})$};
\node (2) at (5,0) {$\mathbb{MV}$-mod$(\mathscr{F})$};
\node (3) at (0,2) {$\mathbb{L}_{u}$-mod$(\mathscr{E})$};
\node (4) at (5,2) {$\mathbb{MV}$-mod$(\mathscr{E})$};
\draw[->] (1) to node [below,midway] {$\Gamma_{\mathscr{F}}$} (2);
\draw[->] (3) to node [left, midway] {$f^{\ast}$} (1);
\draw[->] (4) to node [right, midway] {$f^*$} (2);
\draw[->] (3) to node [above,midway] {$\Gamma_{\mathscr{E}}$} (4);
\end{tikzpicture}
\end{center}

\begin{center} 
\begin{tikzpicture}
\node (1) at (0,0) {$\mathbb{MV}$-mod$(\mathscr{F})$};
\node (2) at (5,0) {$\mathbb{L}_{u}$-mod$(\mathscr{F})$};
\node (3) at (0,2) {$\mathbb{MV}$-mod$(\mathscr{E})$};
\node (4) at (5,2) {$\mathbb{L}_{u}$-mod$(\mathscr{E})$};
\draw[->] (1) to node [below,midway] {$L_{\mathscr{F}}$} (2);
\draw[->] (3) to node [left, midway] {$f^*$} (1);
\draw[->] (4) to node [right, midway] {$f^*$} (2);
\draw[->] (3) to node [above,midway] {$L_{\mathscr{E}}$} (4);
\end{tikzpicture}
\end{center}

The commutativity of these diagrams follows from the fact that all the constructions that we used to build the functors $\Gamma$ and $L$ are geometric (i.e., only involving finite limits and colimits) and hence preserved by the inverse image functors of geometric morphisms. 

We have therefore proved the following

\begin{theorem}\label{main}
The functors $L_{\mathscr{E}}$ and $\Gamma_{\mathscr{E}}$ defined above yield a Morita-equivalence between the theories $\mathbb{MV}$ and ${\mathbb L}_{u}$. In particular, $\mathbf{Set}[\mathbb{MV}]\simeq \mathbf{Set}[\mathbb{L}_{u}]$.
\end{theorem}

\begin{remarks}\label{rem2}
\begin{enumerate}[(a)]
\item We have observed in section \ref{interpretation} that the theories $\mathbb{MV}$ and $\mathbb{L}_{u}$ are not bi-interpretable (in the sense that the geometric syntactic categories $\mathscr{C}_{\mathbb{MV}}$ and $\mathscr{C}_{\mathbb{L}_{u}}$ are not equivalent). On the other hand, we have just proved that the $\infty$-pretopos completions $\mathbf{Set}[\mathbb{MV}]$ of $\mathscr{C}_{\mathbb{MV}}$ and $\mathbf{Set}[\mathbb{L}_{u}]$ of $\mathscr{C}_{\mathbb{L}_{u}}$ are equivalent (by Proposition D3.1.12 \cite{El}, the classifying topos $\Set[{\mathbb T}]$ of a geometric theory is equivalent to the $\infty$-pretopos completion of the geometric syntactic category $\mathscr{C}_{\mathbb T}\hookrightarrow \Set[{\mathbb T}]$ of $\mathbb T$). Now, the objects of the $\infty$-pretopos completion $\Set[{\mathbb T}]$ of the syntactic category $\mathscr{C}_{\mathbb T}$ of a geometric theory $\mathbb T$ are formal quotients of infinite coproducts of objects of $\mathscr{C}_{\mathbb T}$ by equivalence relations in $\Set[{\mathbb T}]$ (cf. the proof of Proposition D1.4.12(iii) \cite{El}). In our particular case, the object $G$ of $\mathbf{Set}[\mathbb{MV}]$ which corresponds to the object $\{x. \top\}$ of $\mathbf{Set}[\mathbb{L}_{u}]$ under the equivalence $\mathbf{Set}[\mathbb{MV}]\simeq \mathbf{Set}[\mathbb{L}_{u}]$ of Theorem \ref{main} can be described as follows. For any natural number $n\geq 1$, let $\phi_{n}(x_{1}, \ldots, x_{n})$ be the formula $\mathbin{\mathop{\textrm{ $\bigwedge$}}\limits_{i\in \{1, \ldots, n-1\}}} x_{i}\oplus x_{i+1}=x_{i}$ over $\Sigma_{\mathbb{MV}}$ asserting that $(x_{1}, \ldots, x_{n})$ is a $n$-good sequence, and let $R$ be the equivalence relation on the coproduct $\mathbin{\mathop{\textrm{ $\coprod$}}\limits_{n\geq 1}}\phi_{n}(x_{1}, \ldots, x_{n})$ defined in section \ref{interpretation}. Then $G$ is isomorphic to the formal quotient of the product $(\mathbin{\mathop{\textrm{ $\coprod$}}\limits_{n\geq 1}}\phi_{n}(x_{1}, \ldots, x_{n}))\slash R\times (\mathbin{\mathop{\textrm{ $\coprod$}}\limits_{n\geq 1}}\phi_{n}(x_{1}, \ldots, x_{n}))\slash R$ by the equivalence relation used for defining the Grothendieck group associated to a cancellative abelian monoid. From this representation of $G$, it is straightforward to derive an expression for $G$ as a formal quotient of an infinite coproduct of formulae in $\mathscr{C}_{\mathbb{MV}}$.  
      
\item We could have alternatively proved that the classifying toposes $\mathbf{Set}[\mathbb{MV}]$ and $\mathbf{Set}[\mathbb{L}_{u}]$ are equivalent by first showing that the theories $\mathbb{MV}$ and $\mathbb{L}_{u}$ are of presheaf type (i.e., classified by a presheaf topos) and then appealing to the classical Mundici's equivalence (the fact that the theory $\mathbb{MV}$ is classified by a presheaf topos is straightforward, it being algebraic, while the fact that $\mathbb{L}_{u}$ is of presheaf type can be proved by using the methods of \cite{OCPT}, cf. section 8.7 therein). Indeed, two theories of presheaf type are Morita-equivalent if and only if they have equivalent categories of set-based models (this immediately follows from the fact, established in \cite{OC}, that for any theory $\mathbb T$ of presheaf type, its classifying topos can be canonically represented as the category $[\textrm{f.p.} {\mathbb T}\textrm{-mod}(\Set), \Set]$ of set-valued functors on the category of finitely presentable $\mathbb T$-models).  
\end{enumerate}
\end{remarks}

\section{Applications}\label{s:applications}

\subsection{Sheaf-theoretic Mundici's equivalence} 

We have defined above, for every Grothendieck topos $\mathscr{E}$, a categorical equivalence between the category of models of ${\mathbb L}_{u}$ in $\mathscr{E}$ and the category of models of $\mathbb{MV}$ in $\mathscr{E}$, which is natural in $\mathscr{E}$. By specializing this result to toposes $\Sh(X)$ of sheaves on a topological space $X$, we shall obtain a sheaf-theoretic generalization of Mundici's equivalence.

The category of models of the theory $\mathbb{MV}$ in the topos $\Sh(X)$ is isomorphic to the category  
$\Sh_{\mathbb{MV}}(X)$ whose objects are the sheaves $F$ on $X$ endowed with an MV-algebra structure on each set $F(U)$ (for an open set $U$ of $X$) in such a way that the maps $F(i_{U, V}):F(U)\to F(V)$ corresponding to inclusions of open sets $i_{U, V}:V\subseteq U$ are MV-algebra homomorphisms, and whose arrows are the natural transformations between them which are pointwise MV-algebra homomorphisms. Indeed, the evaluation functors $ev_{U}:\Sh(X)\to \Set$ (for each open set $U$ of $X$) preserve finite limits and hence preserve and jointly reflect models of the theory $\mathbb{MV}$. 

The category of models of ${\mathbb L}_{u}$ in $\Sh(X)$ is isomorphic to the category $\Sh_{{\mathbb L}_{u}}(X)$ whose objects are the sheaves $F$ on $X$ endowed with an abelian $\ell$-group with unit structure on each set $F(U)$ (for an open set $U$ of $X$) in such a way that the maps $F(i_{U, V}):F(U)\to F(V)$ corresponding to inclusions of open sets $i_{U, V}:V\subseteq U$ are $\ell$-group unital homomorphisms and for each point $x$ of $X$ the canonically induced $\ell$-group structure on the stalks $F_{x}$ is an $\ell$-group with strong unit, and whose arrows are the natural transformations between them which are pointwise abelian $\ell$-group homomorphisms. Indeed, the stalk functors $(-)_{x}:\Sh(X)\to \Set$ (for each point $x$ of $X$) are geometric and jointly conservative, and hence they preserve and jointly reflect models of the theory ${\mathbb L}_{u}$ (cf. Corollary D1.2.14 \cite{El}).   

The two functors $\Gamma_{\Sh(X)}$ and $L_{\Sh(X)}$ defining the equivalence can be described as follows: $\Gamma_{\Sh(X)}$ sends any sheaf $F$ in $\Sh_{{\mathbb L}_{u}}(X)$ to the sheaf $\Gamma_{\Sh(X)}(F)$ on $X$ sending every open set $U$ of $X$ to the MV-algebra given by the unit interval in the $\ell$-group $F(U)$, and it acts on arrows in the obvious way. In the converse direction, $L_{\Sh(X)}$ assigns to any sheaf $G$ in $\Sh_{{\mathbb{MV}}}(X)$ the sheaf $L_{\Sh(X)}(G)$ on $X$ whose stalk at any point $x\in X$ is equal to the $\ell$-group corresponding via Mundici's equivalence to the MV-algebra $G_{x}$. 

The naturality in $\mathscr{E}$ of our Morita-equivalence implies in particular that the resulting equivalence
\[
\tau_{X}:\Sh_{\mathbb{MV}}(X)\simeq \Sh_{{\mathbb L}_{u}}(X)
\]
is natural in $X$ (recall that any continuous map $f:X\to Y$ induces a geometric morphism $\Sh(f):\Sh(X)\to \Sh(Y)$ such that $\Sh(f)^{\ast}$ is the inverse image functor on sheaves along $f$). In particular, by taking $X$ to be the one-point space, we obtain that, at the level of stalks, $\tau_{Y}$ acts as the classical Mundici's equivalence (indeed, the geometric morphism $\Set \to \Sh(X)$ corresponding to a point $x:1\to X$ of $X$ has as inverse image precisely the stalk functor at $x$).

Summarizing, we have the following result.

\begin{cor}
Let $X$ be a topological space. Then, with the above notation, we have a categorical equivalence
\[
\tau_{X}:\Sh_{\mathbb{MV}}(X)\simeq \Sh_{{\mathbb L}_{u}}(X)
\]
sending any sheaf $F$ in $\Sh_{{\mathbb L}_{u}}(X)$ to the sheaf $\Gamma_{\Sh(X)}(F)$ on $X$ sending every open set $U$ of $X$ to the MV-algebra given by the unit interval in the $l$-group $F(U)$, and any sheaf $G$ in $\Sh_{{\mathbb{MV}}}(X)$ to the sheaf $L_{\Sh(X)}(G)$ in $\Sh_{{\mathbb L}_{u}}(X)$ whose stalk at any point $x$ of $X$ is the $\ell$-group corresponding to the MV-algebra $G_{x}$ under Mundici's equivalence.

The equivalence $\tau_{X}$ is natural in $X$, in the sense that for any continuous map $f:X\to Y$ of topological spaces, the diagram
\[  
\xymatrix {
\Sh_{\mathbb{MV}}(Y)  \ar[d]^{j_{f}} \ar[r]^{\tau_{Y}} & \Sh_{{\mathbb L}_{u}}(Y) \ar[d]^{i_{f}} \\
\Sh_{\mathbb{MV}}(X) \ar[r]^{\tau_{X}} & \Sh_{{\mathbb L}_{u}}(X) } 
\]
commutes, where $i_{f}:\Sh_{\mathbb{MV}}(Y) \to \Sh_{\mathbb{MV}}(X)$ and $j_{f}:\Sh_{{\mathbb L}_{u}}(Y) \to \Sh_{{\mathbb L}_{u}}(X)$ are the inverse image functors on sheaves along $f$.

Moreover, $\tau_{X}$ acts, at the level of stalks, as the classical Mundici's equivalence. 
\end{cor}

\subsection{Correspondence between geometric extensions}

The Morita-equivalence between the theories $\mathbb{MV}$ and $\mathbb{L}_{u}$ established above allows us to transfer properties and results between the two theories according to the `bridge technique' of \cite{OC10}. More specifically, for any given topos-theoretic invariant $I$ one can attempt to build a `bridge' yielding a logical relationship between the two theories by using as `deck' the given equivalence of toposes and as `arches' appropriate site characterizations of $I$.  

For instance, we can use the invariant notion of subtopos  (recall that a subtopos of a given topos is an isomorphism class of geometric inclusions to that topos) to establish a relationship between the quotients (in the sense of section \ref{Morita}) of the two theories:

\[  
\xymatrix {
 & & \Sh(\mathscr{C}_{\mathbb{MV}}, J_{\mathbb{MV}}) \simeq \Sh(\mathscr{C}_{\mathbb{L}_{u}}, J_{\mathbb{L}_u}) \ar@/^12pt/@{--}[drr] & & \\
(\mathscr{C}_{\mathbb{MV}},  J_{\mathbb{MV}})  \ar@/^12pt/@{--}[urr] & & & & (\mathscr{C}_{\mathbb{L}_{u}}, J_{\mathbb{L}_u})}
\]

In fact, the duality theorem of \cite{OC6} (cf. section \ref{Morita}) provides the appropriate characterizations of the notion of subtopos in terms of the syntax of the two theories. This yields at once the following
  
\begin{theorem}\label{restrictions}
Every quotient of the theory $\mathbb{MV}$ is Morita-equivalent to a quotient of the theory $\mathbb{L}_u$, and conversely. These Morita-equivalences are the restrictions of the one between $\mathbb{MV}$ and $\mathbb{L}_u$ of Theorem \ref{main}. 
\end{theorem}

This theorem would be trivial if the two theories $\mathbb{MV}$ and $\mathbb{L}_u$ were bi-interpretable, but we proved in section \ref{interpretation} that this is not the case. The unifying power of the notion of classifying topos allow us to obtain a syntactic result by arguing semantically.

Given this result, it is natural to wonder whether there exists an effective means for obtaining, starting from a given quotient of either ${\mathbb L}_{u}$ or ${\mathbb{MV}}$, an explicit axiomatization of the quotient corresponding to it as in the theorem.

To address this issue, we observe that any pair of syntactic sites of definitions of the two classifying toposes such that the given equivalence of toposes is induced by a morphisms between them yields an explicit means for obtaining the quotient theory corresponding to a given quotient of the theory represented by the domain site; indeed, the correspondence between the quotient sites can be described directly in terms of the morphism, and one can exploit the correspondence between sets of additional sequents and Grothendieck topologies containing the given one established in \cite{OC6} to turn axiomatizations of quotient theories into quotient sites and conversely.

As we have seen in section \ref{interpretation}, the interpretation functor $I:\mathscr{C}_{\mathbb{MV}}\to \mathscr{C}_{{\mathbb L}_{u}}$ induces the equivalence of classifying toposes $\Sh(\mathscr{C}_{\mathbb{MV}}, J_{\mathbb{MV}})\simeq \Sh(\mathscr{C}_{{\mathbb L}_{u}}, J_{{\mathbb L}_{u}})$ (cf. Remark \ref{remint}). We thus obtain the following 

\begin{pro}\label{produal}
Let $\mathbb S$ be a quotient of the theory $\mathbb{MV}$. Then the quotient of ${\mathbb L}_{u}$ corresponding to $\mathbb S$ as in Theorem \ref{restrictions} can be described as the quotient of ${\mathbb L}_{u}$ obtained by adding all the sequents of the form $I(\sigma)$ where $\sigma$ ranges over all the axioms of $\mathbb S$.
\end{pro}

To address the converse direction, we consider the natural way of representing the classifying topos of ${\mathbb L}_{u}$ as a subtopos of the classifying topos of ${\mathbb L}$ (cf. \cite{OC6} for the general technique of calculating the Grothendieck topology corresponding to a given quotient of a theory classified by a presheaf topos). 

By the duality theorem,  the classifying topos for the theory ${\mathbb L}_{u}$ can be represented in the form $\Sh({\textrm{f.p.} {\mathbb L}\textrm{-mod}(\Set)}^{\textrm{op}}, J) \hookrightarrow [\textrm{f.p.} {\mathbb L}\textrm{-mod}(\Set), \Set]$ for a unique topology $J$ (recall that every Horn theory is classified by the topos of covariant set-valued functors on its category of finitely presentable models). It therefore naturally arises the question of whether the equivalence of classifying toposes 
\[
\Sh({\textrm{f.p.} {\mathbb L}\textrm{-mod}(\Set)}^{\textrm{op}}, J)\simeq [\mathbf{MV}_{f.p.}, \Set]
\]
is induced by a morphism of sites $\mathbf{MV}_{f.p.} \to \textrm{f.p.} {\mathbb L}\textrm{-mod}(\Set)$, that is if the $\ell$-groups corresponding to finitely presented MV-algebras are all finitely presented as $\ell$-groups with unit. The answer to this question is positive, and will be provided in the next section.

\subsection{Finitely presented abelian $\ell$-groups with strong unit}

The following result is probably known by specialists but we give a proof as we have not found one in the literature.

\begin{pro}\label{profp}
The finitely presentable abelian $\ell$-groups with strong unit are exactly the abelian $\ell$-groups with strong unit which are finitely presentable (equivalently, finitely presented) as abelian $\ell$-groups with unit.
\end{pro}

\begin{proof}
Recall from chapter $2$ of $\cite{Steinberg}$ that the absolute value $|x|$ of an element $x$ of an abelian $\ell$-group with unit $(G, +, -, \leq, \inf, \sup,0)$ is the element $\sup(x, -x)$, that $|x|\geq 0$ for all $x\in G$, that $|x|=|-x|$ for all $x\in G$ and that the triangular inequality $|x+y|\leq |x|+|y|$ holds for all $x, y\in G$. These properties easily imply that for any abelian $\ell$-group with unit ${\cal G}:=(G, u)$ with generators $x_{1}, \ldots, x_{n}$, if for every $i\in \{1, \ldots, n\}$ there exists a natural number $k_{i}$ such that $|x_{i}|\leq k_{i}u$ then the unit $u$ is strong for $\cal G$ (one can prove by induction on the structure $t_{{\cal G}}(x_{1}, \ldots, x_{n})$ of the elements of $G$). Now, it is immediate to see that for any finitely presented $\ell$-group with unit $(G, u)$, any abelian $\ell$-group with strong unit $(H, v)$ and any abelian $\ell$-group unital homomorphism $f:(G, u)\to (H, v)$ there exists an abelian $\ell$-group with strong unit $(G', u')$ and $\ell$-group unital homomorphisms $h:(G, u)\to (G', u')$ and $g:(G', u')\to (H, v)$ such that $f=g\circ h$. Indeed, given generators $x_{1}, \ldots, x_{n}$ for $G$, since $v$ is a strong unit for $H$ there exists for each $i\in \{1, \ldots, n\}$ a natural number $k_{i}$ such that $|f(x_{i})|\leq k_{i}$; it thus suffices to take $G'$ equal to the quotient of $G$ by the congruence generated by the relations $|x_{i}|\leq k_{i}$ for $i\in \{1, \ldots, n\}$ and $u'=u$.

The fact that every homomorphism from a finitely presented abelian $\ell$-group with unit to an $\ell$-group with unit $(H, v)$ factors through a homomorphism from a $\ell$-group with strong unit to $(H, v)$ clearly implies that every abelian $\ell$-group with strong unit can be expressed as a filtered colimit of abelian $\ell$-groups with strong unit which are finitely presented as abelian $\ell$-groups with unit. Since a retract of a finitely presented abelian $\ell$-group is again finitely presented, we can conclude that every abelian $\ell$-group with strong unit which is finitely presentable in the category of abelian $\ell$-groups with strong unit is finitely presented as an abelian $\ell$-group with unit. This concludes the proof of the proposition, as the other direction is trivial.   
\end{proof}

We shall now proceed to give a syntactic description of the category of finitely presentable abelian $\ell$-groups with strong unit. To this end, we recall from \cite{OCS} the following notions.

Let $\mathbb T$ be a geometric theory over a signature $\Sigma$ and $\phi(\vec{x})$ a geometric formula-in-context over $\Sigma$. Then $\phi(\vec{x})$ is said to be \emph{$\mathbb T$-irreducible} if for any family $\{\theta_{i} \textrm{ | } i\in I\}$ of $\mathbb T$-provably functional geometric formulae $\{\vec{x_{i}}, \vec{x}.\theta_{i}\}$ from $\{\vec{x_{i}}. \phi_{i}\}$ to $\{\vec{x}. \phi\}$ such that $\phi \vdash_{\vec{x}} \mathbin{\mathop{\textrm{\huge $\vee$}}\limits_{i\in I}}(\exists \vec{x_{i}})\theta_{i}$ is provable in $\mathbb T$, there exist $i\in I$ and a $\mathbb T$-provably functional geometric formula $\{\vec{x}, \vec{x_{i}}. \theta'\}$ from $\{\vec{x}. \phi\}$ to $\{\vec{x_{i}}. \phi_{i}\}$ such that $\phi \vdash_{\vec{x}} (\exists \vec{x_{i}})(\theta' \wedge \theta_{i})$ is provable in $\mathbb T$.

In other words, a formula-in-context $\phi(\vec{x})$ is $\mathbb T$-irreducible if and only if every $J_{\mathbb T}$-covering sieve on the object $\{\vec{x}. \phi\}$ of the geometric syntactic category $\mathscr{C}_{\mathbb T}$ of $\mathbb T$ is trivial.

\begin{theorem}[Corollary 3.15 \cite{OCS}]\label{thmpresheaf}
Let $\mathbb T$ be a geometric theory over a signature $\Sigma$. Then $\mathbb T$ is classified by a presheaf topos if and only if there exists a collection $\cal F$ of $\mathbb T$-irreducible formulae-in-context over $\Sigma$ such that for every geometric formula $\{\vec{y}. \psi\}$ over $\Sigma$ there exist objects $\{\vec{x_{i}}. \phi_{i}\}$ in $\cal F$ as $i$ varies in $I$ and $\mathbb T$-provably functional geometric formulae $\{\vec{x_{i}}, \vec{y}.\theta_{i}\}$ from $\{\vec{x_{i}}. \phi_{i}\}$ to $\{\vec{y}. \psi\}$ such that $\psi \vdash_{\vec{y}} \mathbin{\mathop{\textrm{\huge $\vee$}}\limits_{i\in I}}(\exists \vec{x_{i}})\theta_{i}$ is provable in $\mathbb T$.
\end{theorem}

The theorem implies, by the Comparison Lemma, that if $\mathbb T$ is classified by a presheaf topos then its classifying topos $\Sh(\mathscr{C}_{\mathbb T}, J_{\mathbb T})$ is equivalent to the topos $[{\mathscr{C}^{irr}_{\mathbb T}}^{\textrm{op}}, \Set]$, where ${\mathscr{C}^{irr}_{\mathbb T}}$ is the full subcategory of $\mathscr{C}_{\mathbb T}$ on the $\mathbb T$-irreducible formulae, and hence that this latter category is dually equivalent to the category of finitely presentable models of $\mathbb T$ via the equivalence sending any such formula $\{\vec{x}. \phi\}$ to the model of $\mathbb T$ which it presents. In fact, such model corresponds to the geometric functor $\mathscr{C}_{\mathbb T}\to \Set$ represented by the formula $\{\vec{x}. \phi\}$ (cf. section \ref{interpretation}) and hence it admits the following syntactic description: its underlying set is given by $Hom_{\mathscr{C}_{\mathbb T}}(\{\vec{x}. \phi\}, \{z. \top\})$ and the order and operations are the obvious ones.   

On the other hand, the category of finitely presented MV-algebras is well-known to be dual to the algebraic syntactic category $\mathscr{C}^{alg}_{\mathbb{MV}}$ of the theory $\mathbb{MV}$ (cf. section \ref{geom}). 
  
We can thus conclude that, even though the theories $\mathbb{MV}$ and ${\mathbb L}_{u}$ are not bi-interpretable, there exists an equivalence of categories between the category $\mathscr{C}^{alg}_{\mathbb{MV}}$ and the category $\mathscr{C}^{irr}_{{\mathbb L}_{u}}$:

\begin{theorem}\label{synequiv}
With the notation above, we have an equivalence of categories
\[
\mathscr{C}^{alg}_{\mathbb{MV}} \simeq \mathscr{C}^{irr}_{{\mathbb L}_{u}}
\]
representing the syntactic counterpart of the equivalence of categories
\[
\mathcal{MV}_{f.p.}\simeq \textrm{f.p.} {{\mathbb L}_{u}}\textrm{-mod}(\Set).
\]
The former equivalence is the restriction to $\mathscr{C}^{alg}_{\mathbb{MV}}$ of the interpretation of the theory $\mathbb{MV}$ into the theory ${\mathbb L}_{u}$ defined in section \ref{interpretation}.  
\end{theorem}

\begin{proof}
In view of the arguments preceding the statement of the theorem, it remains to prove that the syntactic equivalence $\mathscr{C}^{alg}_{\mathbb{MV}} \simeq \mathscr{C}^{irr}_{{\mathbb L}_{u}}$ induced by the equivalence of classifying toposes is the restriction of the interpretation functor $I$ defined in section \ref{interpretation}. To this end, it suffices to notice that for any formula $\phi(\vec{x})$ in $\mathscr{C}_{\mathbb{MV}}$, by Corollary \ref{models of MV}, for any abelian $\ell$-group with strong unit $\cal G$ the interpretation of the formula $I(\{\vec{x}.\phi\})$ in $\cal G$ is naturally in bijection with the interpretation of the formula $\phi(\vec{x})$ in the associated MV-algebra $[0, u_{\cal G}]$; therefore, if the formula $\phi(\vec{x})$ presents a MV-algebra $\cal A$ then the $\ell$-group ${\cal G}_{\cal A}$ corresponding to $\cal A$ via Mundici's equivalence satisfies the universal property of the abelian $\ell$-group presented by the formula $I(\{\vec{x}.\phi\})$. Indeed, for any $\ell$-group with strong unit ${\cal G}=(G, u)$, the interpretation of the formula $I(\{\vec{x}.\phi\})$ in $\cal G$ is by definition of $I$ in natural bijection with the interpretation of the formula $\phi(\vec{x})$ in the MV-algebra $[0, u_{\cal G}]$, which is in turn in bijection with the MV-algebra homomorphisms $A\to [0, u_{\cal G}]$ and hence, by Mundici's equivalence, with the abelian $\ell$-group unital homomorphisms ${\cal G}_{\cal A} \to {\cal G}$.   
\end{proof}

\begin{obs}
The formula-in-context $\{x. \top\}$ is clearly not ${\mathbb L}_{u}$-irreducible, and in fact we proved in section \ref{interpretation} that it is not in the image of the interpretation functor $I:\mathscr{C}_{\mathbb{MV}}\to \mathscr{C}_{{\mathbb L}_{u}}$.
\end{obs}

As a corollary of the theorem and Proposition \ref{profp}, we obtain the following result.

\begin{cor}
The finitely presentable $\ell$-groups with strong unit are exactly the finitely presentable $\ell$-groups with unit which are presented by a ${\mathbb L}_{u}$-irreducible formula. The $\ell$-group presented by such a formula $\phi(\vec{x})$ has as underlying set the set $Hom_{\mathscr{C}^{irr}_{{\mathbb L}_{i}}}(\{\vec{x}. \phi\}, \{z. \top\})$ of ${\mathbb L}_{u}$-provable equivalence classes of ${\mathbb L}_{u}$-provably functional geometric formulae from $\{\vec{x}. \phi\}$ to $\{z. \top\}$, and as order and operations the obvious ones. In fact, if we consider $\mathbb L$ as an algebraic theory (i.e., without the predicate $\leq$, which can be defined in terms of the operation inf), we have a canonical isomorphism $Hom_{\mathscr{C}^{irr}_{{\mathbb L}_{i}}}(\{\vec{x}. \phi\}, \{z. \top\})\cong Hom_{\mathscr{C}^{alg}_{\mathbb L}}(\{\vec{x}. \phi\}, \{z. \top\})$; that is, for any ${\mathbb L}_{u}$-provably functional formula $\theta(\vec{x}, z):\{\vec{x}. \phi\}\to \{z. \top\}$ there exists a term $t(\vec{x})$ over the signature of ${\mathbb L}_{u}$ such that the sequent $(\theta(\vec{x}, z)\dashv\vdash_{\vec{x}, z} z=t(\vec{x}))$ is provable in ${\mathbb L}_{u}$. 

Conversely, a formula-in-context $\phi(\vec{x})$ over the signature of ${\mathbb L}_{u}$ presents a ${\mathbb L}_{u}$-model if and only if it is ${\mathbb L}_{u}$-irreducible, or equivalently if there exists a formula $\psi(\vec{x})$ in $\mathscr{C}^{alg}_{\mathbb{MV}}$ such that $I(\{\vec{x}. \psi\})$ is ${\mathbb L}_{u}$-equivalent (in the sense of section \ref{geom}) to $\{\vec{x}.\phi\}$.
\end{cor}

Thanks to Theorem \ref{synequiv}, we can now describe a method for obtaining an axiomatization of the quotient of ${\mathbb{MV}}$ corresponding to a given quotient of the theory ${\mathbb L}_{u}$ as in Theorem \ref{restrictions} (recall that the converse direction was already addressed by Proposition \ref{produal}). Indeed, since the classifying toposes of $\mathbb{MV}$ (resp. of ${\mathbb L}_{u}$) can be represented in the form $[{\mathscr{C}^{alg}_{\mathbb{MV}}}^{\textrm{op}}, \Set]$ (resp. in the form $[{\mathscr{C}^{irr}_{{\mathbb L}_{u}}}^{\textrm{op}}, \Set]$), by the duality theorem, the quotients of $\mathbb{MV}$ (resp. of ${\mathbb L}_{u}$) are in bijective correspondence with the Grothendieck topologies on the category ${\mathscr{C}^{alg}_{\mathbb{MV}}}$ (resp. on the category $\mathscr{C}^{irr}_{{\mathbb L}_{u}}$); as the categories ${\mathscr{C}^{alg}_{\mathbb{MV}}}$ and ${\mathscr{C}^{irr}_{{\mathbb L}_{u}}}$ are equivalent, the Grothendieck topologies on the two categories correspond to each other bijectively through this equivalence, yielding the desired correspondence between the quotient theories. Specifically, any Grothendieck topology $K$ on ${\mathscr{C}^{irr}_{{\mathbb L}_{u}}}$ corresponds to the quotient ${\mathbb L}^{K}_{u}$ of ${\mathbb L}_{u}$ consisting of all the geometric sequents of the form $(\psi \vdash_{\vec{y}} \mathbin{\mathop{\textrm{\huge $\vee$}}\limits_{i\in I}}(\exists \vec{x_{i}})\theta_{i})$, where $\{\vec{y}. \psi\}$ and the $\{\vec{x_{i}}. \phi_{i}\}$ are all ${\mathbb L}_{u}$-irreducible formulas, the $\{ \vec{x_{i}}, \vec{y}.\theta_{i}\}$ are geometric formulae over the signature of ${\mathbb L}_{u}$ which are ${\mathbb L}_{u}$-provably functional from $\{\vec{x_{i}}. \phi_{i}\}$ to $\{\vec{y}. \psi\}$ and the sieve generated by the family of arrows $\{[\theta_{i}]:\{\vec{x_{i}}. \phi_{i}\} \to  \{\vec{y}. \psi\} \textrm{ | } i\in I \}$ in ${\mathscr{C}^{irr}_{{\mathbb L}_{u}}}$ generates a $K$-covering sieve. Conversely, every quotient $\mathbb S$ of ${\mathbb L}_{u}$ is syntactically equivalent to a quotient of this form (cf. Theorem \ref{thmpresheaf}); the Grothendieck topology associated to it is therefore the topology on the category ${\mathscr{C}^{irr}_{{\mathbb L}_{u}}}$ generated by the (sieves generated by the) families of arrows $\{[\theta_{i}]:\{\vec{x_{i}}. \phi_{i}\} \to  \{\vec{y}. \psi\} \textrm{ | } i\in I \}$ in ${\mathscr{C}^{irr}_{{\mathbb L}_{u}}}$  such that the sequent $(\psi \vdash_{\vec{y}} \mathbin{\mathop{\textrm{\huge $\vee$}}\limits_{i\in I}}(\exists \vec{x_{i}})\theta_{i})$ is provable in ${\mathbb S}$. An analogous correspondence holds for the theory $\mathbb{MV}$ (with the ${\mathbb L}_{u}$-irreducible formulae being replaced by the formulae in ${\mathscr{C}^{alg}_{\mathbb{MV}}}$). The quotient of $\mathbb{MV}$ corresponding to a given quotient $\mathbb S$ of ${\mathbb L}_{u}$ is thus the quotient of $\mathbb{MV}$ corresponding to the Grothendieck topology on ${\mathscr{C}^{alg}_{\mathbb{MV}}}$ obtained by transferring the Grothendieck topology on ${\mathscr{C}^{irr}_{{\mathbb L}_{u}}}$ associated to $\mathbb S$ along the equivalence $\mathscr{C}^{alg}_{\mathbb{MV}} \simeq \mathscr{C}^{irr}_{{\mathbb L}_{u}}$ of Theorem \ref{synequiv}.

\subsection{Geometric compactness and completeness for ${\mathbb L}_{u}$}\label{logicprop}

The Morita-equivalence between the geometric theory ${\mathbb L}_{u}$ of abelian $\ell$-groups with strong unit and the algebraic theory $\mathbb{MV}$ of MV-algebras implies a form of compactness and completeness for the theory ${\mathbb L}_{u}$, properties which are \emph{a priori} not expected as this theory is infinitary. 

\begin{theorem}
\begin{enumerate}[(i)]
\item For any geometric sequent $\sigma$ in the signature $\Sigma_{{\mathbb L}_{u}}$, $\sigma$ is valid in all abelian $\ell$-groups with strong unit in $\Set$ if and only if it is provable in the theory ${\mathbb L}_{u}$;

\item For any geometric sentences $\phi_{i}$ in the signature $\Sigma_{{\mathbb L}_{u}}$, $\top \vdash \mathbin{\mathop{\textrm{ $\bigvee$}}\limits_{i\in I}}\phi_{i}$ is provable in ${\mathbb L}_{u}$ (equivalently by $(i)$, every abelian $\ell$-group with strong unit in $\Set$ satisfies at least one of the $\phi_{i}$) if and only if there exists a finite subset $J\subseteq I$ such that the sequent $\top \vdash \mathbin{\mathop{\textrm{ $\bigvee$}}\limits_{i\in J}}\phi_{i}$ is provable in ${\mathbb L}_{u}$ (equivalently by $(i)$, every abelian $\ell$-group with strong unit in $\Set$ satisfies at least one of the $\phi_{i}$ for $i\in J$).
\end{enumerate}
\end{theorem}

\begin{proof}
$(i)$ This follows from the fact that every theory classified by a presheaf topos has enough (finitely presentable) set-based models. 

$(ii)$ The fact that the theory $\mathbb{MV}$ is algebraic implies that its classifying topos is coherent; in particular, its terminal object is a compact object of the topos $\Set[\mathbb{MV}]$, in the sense that every covering of it in the topos by a family of subobjects admits a finite subcovering. On the other hand, the terminal object in the classifying topos $\Set[{\mathbb L}_{u}]\simeq \Sh(\mathscr{C}_{{\mathbb L}_{u}}, J_{{\mathbb L}_{u}})$ is the image of the object $\{[].\top\}$ of $\mathscr{C}_{{\mathbb L}_{u}}$ under the Yoneda embedding $y:\mathscr{C}_{{\mathbb L}_{u}} \hookrightarrow \Sh(\mathscr{C}_{{\mathbb L}_{u}}, J_{{\mathbb L}_{u}})$, and its subobjects are precisely the images under $y$ of the subobjects of $\{[].\top\}$ in $\mathscr{C}_{{\mathbb L}_{u}}$, that is of the ${\mathbb L}_{u}$-provable equivalence classes of geometric sentences over the signature $\Sigma_{{\mathbb L}_{u}}$. The thesis thus follows from the fact that a family of such subobjects $\{ y(\{[].\phi_{i}\}) \mid i\in I\}$ covers $y(\{[]. \top\})$ in $\Set[{\mathbb L}_{u}]$ if and only if the sequent $\top \vdash \mathbin{\mathop{\textrm{ $\bigvee$}}\limits_{i\in I}}\phi_{i}$ is provable in ${\mathbb L}_{u}$. 
\end{proof}
  
Notice that this proof represents an application of the `bridge technique' in the context of the equivalence $\Sh(\mathscr{C}_{{\mathbb L}_{u}}, J_{{\mathbb L}_{u}})\simeq [\mathbf{MV}_{f.p.}, \Set]$ and the topos-theoretic concept of terminal object.

\section{Appendix: topos-theoretic background}\label{s:background}

In this section we introduce the basic topos-theoretic background necessary to understand the paper, for the benefit of readers who are not familiar with topos theory. Classical references on the subject, in increasing order of sophistication, are \cite{Goldblatt}, \cite{MM} and \cite{El}.

\subsection{Geometric theories and categorical semantics}\label{geom}

A \emph{geometric theory} over a first-order signature $\Sigma$ is a theory whose axioms can be presented in the form $\phi\vdash_{\vec{x}}\psi$, where $\phi$ and $\psi$ are \emph{geometric formulae}, i.e. formulae with finite number of free variables in the context $\vec{x}$ built up from atomic formulae over $\Sigma$ by only using finitary conjunctions, infinitary disjunctions and existential quantifications.

A \emph{finitary algebraic theory} is a geometric theory over a signature without relation symbols such that all its axioms are of the form $(\top \vdash_{\vec{x}} t=s)$, where $t$ and $s$ are terms over $\Sigma$ in the context $\vec{x}$.

Recall that a Grothendieck topos is (a category equivalent to) a category $\Sh(\mathscr{C}, J)$ of sheaves on a site $(\mathscr{C}, J)$, and that a geometric morphism $f:\mathscr{F}\to \mathscr{E}$ of Grothendieck toposes is a pair of adjoint functors such that the left adjoint $f^{\ast}:\mathscr{E}\to \mathscr{F}$ preserves finite limits. In particular, the category $\Set$ of sets and functions between them is a Grothendieck topos.

One can define the notion of model of a geometric theory $\mathbb T$ in a Grothendieck topos $\mathscr{E}$ by generalizing the classical Tarskian definition of model of a first-order theory (in sets). 

\begin{definition}
Let $\mathscr{E}$ be a topos and $\Sigma$ be a (possibly multi-sorted) first-order signature. A $\Sigma$-\textit{structure} $M$ in $\mathscr{E}$ is specified by the following data:
\begin{enumerate}
\item[(i)] a function assigning to each sort $A$ of $\Sigma$, an object $MA$ of $\mathscr{E}$. This function is extended to finite strings of sort by defining $M(A_{1},...,A_{n})=MA_{1}\times ...\times MA_{n}$ (and setting $M([])$, where [] denotes the empty string, equal to the terminal object 1 of $\mathscr{E}$);
\item[(ii)] a function assigning to each function symbol $f:A_{1}...A_{n}\rightarrow B$ in $\Sigma$ an arrow $Mf:M(A_{1},...,A_{n})\rightarrow MB$ in $\mathscr{E}$;
\item[(iii)] a function assigning to each relation symbol $R\rightarrowtail A_{1}...A_{n}$ in $\Sigma$ a subobject $MR\rightarrowtail M(A_{1},...A_{n})$ in $\mathscr{E}$.
\end{enumerate}

The $\Sigma$-structures in $\mathscr{E}$ are the objects of a category $\Sigma$-$\mathbf{Str}(\mathscr{E})$ whose arrows are the $\Sigma$-structure homomorphisms. Such homomorphisms $h:M\rightarrow N$ are specified by a collection of arrows $h_{A}:MA\rightarrow NA$ in $\mathscr{E}$, indexed by the sorts of $\Sigma$. For more details see section D1.1\cite{El}.
\end{definition}

Let $\mathscr{E}$ and $\mathscr{F}$ be toposes. Any functor $T:\mathscr{E}\rightarrow \mathscr{F}$ which preserves finite products and monomorphisms induces a functor $\Sigma$-$\mathbf{Str}(T):\Sigma$-$\mathbf{Str}(\mathscr{E})\rightarrow \Sigma$-$\mathbf{Str}(\mathscr{F})$ in the obvious way.

Let $M$ be a  $\Sigma$-structure  in $\mathscr{E}$. Any first order formula $\vec{x}.\phi$ over $\Sigma$ will be interpreted as a subobject $[[\vec{x}.\phi]]_{M}\rightarrowtail M(A_{1},...A_{n})$; this interpretation is defined recursively on the structure of the formula.

\begin{definition}
Let $M$ be a $\Sigma$-structure in a topos $\mathscr{E}$.
\begin{itemize}
\item[(a)] If $\sigma=(\phi\vdash_{\vec{x}}\psi)$ is a first-order sequent over $\Sigma$, we say $\sigma$ is \textit{satisfied in} $M$ (and write $M\models \sigma$) if $[[\vec{x}.\phi]]_{M}\leq [[\vec{x}.\psi]]_{M}$ in the lattice $Sub_{\mathscr{E}}(M(A_{1},...,A_{n}))$ of subobjects of $M(A_{1},...,A_{n})$ in $\mathscr{E}$.
\item[(b)] If $\mathbb{T}$ is a geometric theory over $\Sigma$, we say $M$ is a \textit{model} of $\mathbb{T}$ (and write $M\models \mathbb{T}$) if all the axioms of $\mathbb{T}$ are satisfied in $M$.
\item[(c)] We write $\mathbb{T}$-mod$(\mathscr{E})$ for the full subcategory of $\Sigma$-$\mathbf{Str}(\mathscr{E})$ whose objects are the models of $\mathbb{T}$.
\end{itemize}
\end{definition}

\subsection{The internal language of a topos}\label{intlanguage}

As a category, a Grothendieck topos has all small limits and colimits, as well as exponentials and a subobject classifier. It can thus be considered as a mathematical universe in which one can perform all the usual (bounded) set-theoretic constructions. More specifically, one can attach to any topos $\mathscr{E}$ a canonical signature $\Sigma_\mathscr{E}$, called its \emph{internal language}, having a sort $\name{A}$ for each object $A$ of $\mathscr{E}$, a function symbol $\name{f}:\name{A_{1}} ... \name{A_{n}}\to \name{B}$ for each arrow $f:A_{1}\times \cdots \times A_{n}\to B$ in $\mathscr{E}$ and a relation symbol $\name{R}\mono \name{A_{1}} ... \name{A_{n}}$ for each subobject $R\mono A_{1}\times \cdots \times A_{n}$ in $\mathscr{E}$. There is a tautological $\Sigma_\mathscr{E}$-structure ${\cal S}_\mathscr{E}$ in $\mathscr{E}$, obtained by interpreting each $\name{A}$ as $A$, each $\name{f}$ as $f$ and each $\name{R}$ as $R$. For any objects $A_{1}, \ldots, A_{n}$ of $\mathscr{E}$ and any first-order formula $\phi(\vec{x})$ over $\Sigma_\mathscr{E}$, where $\vec{x}=(x_{1}^{\name{A_{1}}}, \ldots, x_{n}^{\name{A_{n}}})$, the expression $\{\vec{x}\in A_{1}\times \cdots \times A_{n} \mid \phi(\vec{x})\}$ can be given a meaning, namely the interpretation of the formula $\phi(\vec{x})$ in the $\Sigma_\mathscr{E}$-structure ${\cal S}_\mathscr{E}$.

The logic of a topos being in general intuitionistic, any formal proof involving first-order sequents over the signature $\Sigma_\mathscr{E}$ will be valid in the structure ${\cal S}_\mathscr{E}$ provided that the law of excluded middle or any other non-constructive principles are not employed in it.

\subsection{Classifying toposes}

\begin{definition}
Let $\mathbb{T}$ be a geometric theory. A Grothendieck topos $\mathscr{E}$ is said to be a \emph{classifying topos} for $\mathbb T$ if for any Grothendieck topos $\mathscr{F}$, the category of geometric morphisms $\mathscr{F}\to \mathscr{E}$ is equivalent to the category of models of $\mathbb T$ inside $\mathscr{F}$, naturally in $\mathscr{F}$.
\end{definition}

Clearly, a classifying topos of a given geometric theory $\mathbb T$ is unique up to categorical equivalence; we shall denote it by $\Set[{\mathbb T}]$. Every geometric theory has a classifying topos; conversely, every Grothendieck topos is the classifying topos of a geometric theory, albeit not canonically.

Classifying toposes for geometric theories can be built directly through the construction of syntactic sites.

Let $\mathbb{T}$ be a geometric theory and $\phi(\vec{x})$ and $\psi(\vec{y})$ be two formulae in $\mathbb{T}$, where $\vec{x}$ and $\vec{y}$ are contexts of the same type and length. We say that these formulae are $\alpha$-equivalent if and only if
$\psi(\vec{y})$ is obtained from $\phi(\vec{x})$ by an acceptable change of variables, i.e. every free occurrence of $x_i$ is replaced by $y_i$ in $\phi$ and each $x_i$ is free for $y_i$ in $\phi$. We write $\{\vec{x}. \phi\}$ for the $\alpha$-equivalence class of the formula $\phi{\vec{x}}$. 

\begin{definition}
Let $\mathbb T$ be a geometric theory over a signature $\Sigma$. The \emph{geometric syntactic category} $\mathscr{C}_{\mathbb T}$ of $\mathbb T$ has as objects the geometric formulae-in-context $\{\vec{x}.\phi\}$ and as arrows $\{\vec{x}.\phi\}$ and $\{\vec{y}.\psi\}$ the $\mathbb T$-provable equivalence classes $[\theta]$ of geometric formulae $\theta(\vec{x},\vec{y})$, where $\vec{x}$ and $\vec{y}$ are disjoint contexts, which are $\mathbb T$-provably functional from $\{\vec{x}.\phi\}$ to $\{\vec{y}.\psi\}$, i.e. such that the sequents
\begin{itemize}
\item[-]$(\theta\vdash_{\vec{x},\vec{y}}(\phi\wedge\psi))$
\item[-]$(\theta\wedge \theta[\vec{z}/\vec{y}]\vdash_{\vec{x},\vec{y},\vec{z}}(\vec{z}=\vec{y}))$
\item[-]$(\phi\vdash_{\vec{x}}(\exists\vec{y})\theta)$
\end{itemize} 
are provable in $\mathbb{T}$.   

We shall say that two geometric formulae-in-context $\{\vec{x}. \phi\}$ and $\{\vec{y}. \psi\}$, where $\vec{x}$ and $\vec{y}$ are disjoint, are \emph{$\mathbb T$-equivalent} if they are isomorphic objects in the syntactic category $\mathscr{C}_{\mathbb T}$, that is if there exists a geometric formula $\theta(\vec{x}, \vec{y})$ which is $\mathbb T$-provably functional from $\{\vec{x}. \phi\}$ to $\{\vec{y}. \psi\}$ and which moreover satisfies the property that the sequent $(\theta \wedge \theta[\vec{x'}\slash \vec{x}] \vdash_{\vec{x}, \vec{x'}, \vec{y}} \vec{x}=\vec{x'})$ is provable in $\mathbb T$.
\end{definition}
 
We can equip the geometric category $\mathscr{C}_{\mathbb{T}}$ with its canonical coverage, consisting of all sieves generated by small covering families, i.e. families of the form $\{[\vec{x}_i,\vec{y}.\theta_i]\mid i\in I\}$, where $[\theta_i]$ are arrows from $\{\vec{x}_i.\phi_i\}$ to $\{\vec{y}.\psi\}$ in $\mathscr{C}_{\mathbb{T}}$ and the sequent $\psi\vdash_{\vec{y}}\bigvee_{i\in I}(\exists \vec{x}_i\theta_i)$ is provable in $\mathbb{T}$. We call this topology $J_{\mathbb{T}}$. 

The topos $\mathbf{Sh}(\mathscr{C}_{\mathbb{T}},J_{\mathbb{T}})$ satisfies the universal property of the classifying topos for $\mathbb T$.

Classifying toposes for finitary algebraic theories can be built in the following alternative way.

Given a finitary algebraic theory $\mathbb T$, let $\mathscr{C}^{alg}_{\mathbb T}$ be the category whose objects are the finite conjunctions of atomic formulae-in-context (up to $\alpha$-equivalence) over the signature of $\mathbb{\mathbb T}$ and whose arrows $\{\vec{x}. \phi\}\to \{\vec{y}. \psi\}$ (where the contexts $\vec{x}=(x_{1},\ldots, x_{n})$ and $\vec{y}=(y_{1}, \ldots, y_{m})$ are supposed to be disjoint, without loss of generality) are sequences of terms $t_{1}(\vec{x}), \ldots, t_{m}(\vec{x})$ such that the sequent $(\phi \vdash_{\vec{x}} \psi(t_{1}(\vec{x})), \ldots, t_{n}(\vec{x}))$ is provable in $\mathbb T$, modulo the equivalence relation which identifies two such sequences $\vec{t}$ and $\vec{t'}$ precisely when the sequent $(\phi \vdash_{x} \vec{t}(\vec{x})=\vec{t'}(\vec{x}))$ is provable in $\mathbb{T}$.

The category ${\mathscr{C}^{alg}_{\mathbb T}}$ is dual to the category of finitely presented $\mathbb T$-algebras, and the presheaf topos $[{\mathscr{C}^{alg}_{\mathbb T}}^{\textrm{op}}, \Set]$ satisfies the universal property of the classifying topos for $\mathbb T$.

\subsection{Morita-equivalence and the Duality theorem}\label{Morita}

\begin{definition}
Two geometric theories $\mathbb{T}$ and $\mathbb{T}'$ are said to be \textit{Morita-equivalent} if they have equivalent classifying toposes, equivalently if they have equivalent categories of models in every Grothendieck topos $\mathscr{E}$, naturally in $\mathscr{E}$, that is for each Grothendieck topos $\mathscr{E}$ there is an equivalence of categories
\begin{center}
$\tau_{\mathscr{E}}:\mathbb{T}$-mod$(\mathscr{E})\rightarrow \mathbb{T}'$-mod$(\mathscr{E})$
\end{center}
such that for any geometric morphism $f:\mathscr{F}\rightarrow \mathscr{E}$ the following diagram commutes (up to isomorphism):

\begin{center}
\begin{tikzpicture}
\node (1) at (0,0) {$\mathbb{T}$-mod$(\mathscr{F})$};
\node (2) at (5,0) {$\mathbb{T}'$-mod$(\mathscr{F})$};
\node (3) at (0,2) {$\mathbb{T}$-mod$(\mathscr{E})$};
\node (4) at (5,2) {$\mathbb{T}'$-mod$(\mathscr{E})$};
\draw[->] (1) to node [below,midway] {$\tau_{\mathscr{F}}$} (2);
\draw[->] (3) to node [left, midway] {$f^{*}$} (1);
\draw[->] (4) to node [right, midway] {$f^{*}$} (2);
\draw[->] (3) to node [above,midway] {$\tau_{\mathscr{E}}$} (4);
\end{tikzpicture}
\end{center}
\end{definition}

\begin{definition}
A \textit{quotient} of a geometric theory $\mathbb{T}$ is a geometric theory $\mathbb{T}'$ over the same signature such that every axiom of $\mathbb{T}$ is provable in $\mathbb{T}'$ (cf. \cite{OC6}).
\end{definition}

\begin{theorem}[\textbf{Duality theorem}, \cite{OC6}]
Let $\mathbb{T}$ be a geometric theory over a signature $\Sigma$. Then the assignment sending a quotient of $\mathbb{T}$ to its classifying topos defines a bijection between the quotients of $\mathbb{T}$ (considered up to the equivalence which identifies two quotients precisely when they prove the same geometric sequents over their signature)and the subtoposes of the classifying topos $\mathbf{Set}[\mathbb{T}]$ of $\mathbb{T}$.
\end{theorem}

\section{Conclusions and future directions}

By interpreting Mundici's equivalence as a Morita-equivalence between two geometric theories, we have opened up the way for an investigation of this equivalence by using the methods of topos theory.

Our applications show that this topos-theoretic approach is fruitful in bringing new insights on this by now classical subject. Also, these results do not exhaust the range of applicability of the `bridge technique' of \cite{OC10} in the context of this specific Morita-equivalence. Indeed, one can naturally expect the consideration of other topos-theoretic invariant properties or constructions admitting appropriate site characterizations to yield further connections between the two theories.  

On the other hand, the generality of the methods employed indicates that it is reasonable to expect them to be applicable in the context of other equivalences. For instance, it is not hard to see that the extended categorical equivalence of \cite{GalatosTsinakis} can be strengthened to a Morita-equivalence (notice that the concepts defined therein of GMV-algebra and of a pair consisting of an $\ell$-group and a core on it whose image generates the group are both formalizable within geometric logic); also, the equivalence of \cite{DiNolaLettieri} can be interpreted as a Morita-equivalence (the concept of radical of a MV-algebra admits a geometric description).

\vspace{1cm}

\textbf{Acknowledgements:} We thank Vincenzo Marra and Antonio Di Nola for useful discussions on the subject matter of this paper. We are also grateful to Giancarlo Meloni for suggesting us to write down a syntactic expression for the object of $\Set[{\mathbb MV}]$ corresponding to the object $\{x. \top\}$ of $\Set[{\mathbb L}_{u}]$ under the Morita-equivalence $\Set[{\mathbb MV}]\simeq \Set[{\mathbb L}_{u}]$ of Theorem \ref{main} (cf. Remark \ref{rem2}).

\newpage

\emph{Olivia Caramello}

{\small \textsc{Institut des Hautes \'Etudes Scientifiques, 35 route de Chartres, 91440, Bures-sur-Yvette, France}\\
\emph{E-mail address:} \texttt{olivia@ihes.fr}}

\vspace{0.5cm}

\emph{Anna Carla Russo}

{\small \textsc{Dipartimento di Matematica e Informatica, Universit\'a di Salerno, Via Giovanni Paolo II, 84084 Fisciano, Italy}\\
\emph{E-mail address:} \texttt{anrusso@unisa.it}}

\end{document}